\documentclass{amsart}
\usepackage{geometry}
\usepackage{amsmath}
\usepackage{amssymb}
\usepackage[utf8]{inputenc}
\usepackage[english]{babel}
\usepackage{amsthm}
\usepackage{proof}

\usepackage{color}
\definecolor{myurlcolor}{rgb}{0.6,0,0}
\definecolor{mycitecolor}{rgb}{0,0,0.8}
\definecolor{myrefcolor}{rgb}{0,0,0.8}
\usepackage[pagebackref]{hyperref}
\usepackage{hyperref}
\hypersetup{colorlinks,linkcolor=myrefcolor,citecolor=mycitecolor,urlcolor=myurlcolor}

\usepackage[all,2cell,cmtip]{xy}
\UseTwocells
\usepackage[svgnames]{xcolor}

\usepackage{tikz}
\usepackage{tikz-cd}
\usetikzlibrary{backgrounds,circuits,circuits.ee.IEC,shapes,fit,matrix}

\tikzstyle{simple}=[-,line width=2.000]
\tikzstyle{arrow}=[-,postaction={decorate},decoration={markings,mark=at position .5 with {\arrow{>}}},line width=1.100]
\pgfdeclarelayer{edgelayer}
\pgfdeclarelayer{nodelayer}
\pgfsetlayers{edgelayer,nodelayer,main}

\tikzstyle{none}=[inner sep=0pt]

\definecolor{lblue}{rgb}{0,250,255}
\tikzstyle{species}=[circle,fill=yellow,draw=black,scale=1.15]
\tikzstyle{transition}=[rectangle,fill=lblue,draw=black,scale=1.15]
\tikzstyle{inarrow}=[->, >=stealth, shorten >=.03cm,line width=1.5]
\tikzstyle{empty}=[circle,fill=none, draw=none]
\tikzstyle{inputdot}=[circle,fill=purple,draw=purple, scale=.25]
\tikzstyle{inputarrow}=[->,draw=purple, shorten >=.05cm]
\tikzstyle{simple}=[-,draw=purple,line width=1.000]

\newcommand{\N}{\mathbb{N}}

\newcommand{\CMC}{\mathsf{CMC}}

\newcommand{\Set}{\mathsf{Set}}
\newcommand{\Petri}{\mathsf{Petri}}

\newcommand{\Csp}{\mathsf{Csp}}
\newcommand{\CMon}{\mathsf{CommMon}}
\newcommand{\Cat}{\mathsf{Cat}}
\newcommand{\open}{\mathsf{Open}}
\newcommand{\A}{\mathsf{A}}
\newcommand{\B}{\mathsf{B}}
\newcommand{\C}{\mathsf{C}}
\newcommand{\X}{\mathsf{X}}
\newcommand{\J}{\mathsf{J}}
\newcommand{\Q}{\mathsf{Q}}

\newcommand{\Cospan}{\mathbb{C}\mathbf{sp}}
\newcommand{\Open}{\mathbb{O}\mathbf{pen} }

\newcommand{\dRel}{\mathbb{R}\mathbf{el}}

\newcommand{\lA}{\mathbb{A}}
\newcommand{\lB}{\mathbb{B}}
\newcommand{\lC}{\mathbb{D}}
\newcommand{\lD}{\mathbb{D}}

\renewcommand{\hom}{\mathrm{hom}}
\newcommand{\maps}{\colon}

\newcommand{\colim}{\mathrm{colim}}
\newcommand{\Mor}{\mathrm{Mor}}
\newcommand{\Ob}{\mathrm{Ob}}

\newcommand{\define}[1]{{\bf \boldmath{#1}}}

\theoremstyle{plain}

\newtheorem{thm}{Theorem}
\newtheorem{lem}[thm]{Lemma}
\newtheorem{prop}[thm]{Proposition}

\theoremstyle{definition}
\newtheorem{defn}[thm]{Definition}

\begin{document}
	
\begin{center}   
\title{Open Petri Nets}
\maketitle
	{\em John\ C.\ Baez \\}
	\vspace{0.3cm}
	{\small
		Department of Mathematics \\
		University of California \\
		Riverside CA, USA 92521 \\ and \\
		Centre for Quantum Technologies  \\
		National University of Singapore \\
		Singapore 117543  \\    } 
	\vspace{0.4cm}
	{\em Jade Master \\ }
	\vspace{0.3cm}
	{\small Department of Mathematics \\
		University of California \\
		Riverside CA, USA 92521 \\ }
	\vspace{0.3cm}   
	{\small email:  baez@math.ucr.edu, jmast003@ucr.edu\\} 
	\vspace{0.3cm}   
	{\small \today}
	\vspace{0.3cm}   
\end{center}   

\begin{abstract}
\noindent The reachability semantics for Petri nets can be studied using open Petri nets.  For us an `open' Petri net is one with certain places designated as inputs and outputs via a cospan of sets.   We can compose open Petri nets by gluing the outputs of one to the inputs of another.  Open Petri nets can be treated as morphisms of a category $\open(\Petri)$, which becomes symmetric monoidal under disjoint union.   However, since the composite of open Petri nets is defined only up to isomorphism, it is better to treat them as morphisms of a symmetric monoidal \emph{double} category $\Open(\Petri)$.   We describe two forms of semantics for open Petri nets using symmetric monoidal double functors out of $\Open(\Petri)$. The first, an operational semantics, gives for each open Petri net a category whose morphisms are the processes that this net can carry out.  This is done in a compositional way, so that these categories can be computed on smaller subnets and then glued together. The second, a reachability semantics, simply says which markings of the outputs can be reached from a given marking of the inputs.
\end{abstract}
	
\section{Introduction}
	
Petri nets are a simple and widely studied model of computation \cite{GiraultValk,Gorrieri,Peterson}, with generalizations applicable to many forms of modeling \cite{JensenKristensen}.   Recently more attention has been paid to a compositional treatment in which Petri nets can be assembled from smaller `open' Petri nets \cite{BaezPollard,BCEH,BBGM,BMM,BMMS,BMeseguerMS}.   In particular, the reachability problem for Petri nets, which asks whether one marking of a Petri net can be obtained from another via a sequence of transitions, can be studied compositionally \cite{RSO,Congruence,SO}.  
Here we seek to give this line of work a firmer footing in category theory. Petri nets are closely tied to symmetric monoidal categories in two ways.  First, a Petri net $P$ can be seen as a presentation of a free symmetric monoidal category $FP$, with the places and 
transitions of $P$ serving to freely generate the objects and morphisms of $FP$.  We show how to
construct this in Section \ref{sec:presentation}, after reviewing a line of previous work going back to
Meseguer and Montanari \cite{MM}.   In these terms, the reachability problem asks whether there is a morphism from one object of $FP$ to another.

Second, there is a symmetric monoidal category where the objects are sets and the morphisms are equivalence classes of open Petri nets.   We construct this in Section \ref{sec:open_petri}, but the basic idea is very simple.  Here is an open Petri net $P$ from a set $X$ to a set $Y$:
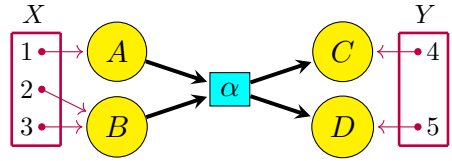
\[
\begin{tikzpicture}
	\begin{pgfonlayer}{nodelayer}
		\node [style=species] (A) at (-4, 0.5) {$A$};
		\node [style=species] (B) at (-4, -0.5) {$B$};
		\node [style=species] (C) at (-1, 0.5) {$C$};
		\node [style=species] (D) at (-1, -0.5) {$D$};
             \node [style=transition] (a) at (-2.5, 0) {$\alpha$}; 
		
		\node [style=empty] (X) at (-5.1, 1) {$X$};
		\node [style=none] (Xtr) at (-4.75, 0.75) {};
		\node [style=none] (Xbr) at (-4.75, -0.75) {};
		\node [style=none] (Xtl) at (-5.4, 0.75) {};
             \node [style=none] (Xbl) at (-5.4, -0.75) {};
	
		\node [style=inputdot] (1) at (-5, 0.5) {};
		\node [style=empty] at (-5.2, 0.5) {$1$};
		\node [style=inputdot] (2) at (-5, 0) {};
		\node [style=empty] at (-5.2, 0) {$2$};
		\node [style=inputdot] (3) at (-5, -0.5) {};
		\node [style=empty] at (-5.2, -0.5) {$3$};

		\node [style=empty] (Y) at (0.1, 1) {$Y$};
		\node [style=none] (Ytr) at (.4, 0.75) {};
		\node [style=none] (Ytl) at (-.25, 0.75) {};
		\node [style=none] (Ybr) at (.4, -0.75) {};
		\node [style=none] (Ybl) at (-.25, -0.75) {};

		\node [style=inputdot] (4) at (0, 0.5) {};
		\node [style=empty] at (0.2, 0.5) {$4$};
		\node [style=inputdot] (5) at (0, -0.5) {};
		\node [style=empty] at (0.2, -0.5) {$5$};		
		
%		\node [style=empty] (Z) at (3, 1) {$Z$};
%		\node [style=none] (Ztr) at (3.25, 0.75) {};
%		\node [style=none] (Ztl) at (2.75, 0.75) {};
%		\node [style=none] (Zbl) at (2.75, -0.75) {};
%		\node [style=none] (Zbr) at (3.25, -0.75) {};
		
	\end{pgfonlayer}
	\begin{pgfonlayer}{edgelayer}
		\draw [style=inarrow] (A) to (a);
		\draw [style=inarrow] (B) to (a);
		\draw [style=inarrow] (a) to (C);
		\draw [style=inarrow] (a) to (D);
		\draw [style=inputarrow] (1) to (A);
		\draw [style=inputarrow] (2) to (B);
		\draw [style=inputarrow] (3) to (B);
		\draw [style=inputarrow] (4) to (C);
		\draw [style=inputarrow] (5) to (D);
		\draw [style=simple] (Xtl.center) to (Xtr.center);
		\draw [style=simple] (Xtr.center) to (Xbr.center);
		\draw [style=simple] (Xbr.center) to (Xbl.center);
		\draw [style=simple] (Xbl.center) to (Xtl.center);
		\draw [style=simple] (Ytl.center) to (Ytr.center);
		\draw [style=simple] (Ytr.center) to (Ybr.center);
		\draw [style=simple] (Ybr.center) to (Ybl.center);
		\draw [style=simple] (Ybl.center) to (Ytl.center);
	\end{pgfonlayer}
\end{tikzpicture}
\]
The yellow circles are places and the blue rectangle is a transition.   The bold arrows from places to transitions and from transitions to places complete the structure of a Petri net.  There are also arbitrary functions from $X$ and $Y$ into the set of places.  These indicate points at which tokens could flow
in or out, making our Petri net `open'.   We write this open Petri net as $P \maps X \nrightarrow Y$
for short.

Given another open Petri net $Q \maps Y \nrightarrow Z$:
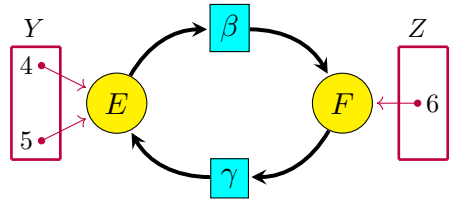
\[
\begin{tikzpicture}
	\begin{pgfonlayer}{nodelayer}
%		\node [style=species] (A) at (-4, 0.5) {$A$};
%		\node [style=species] (B) at (-4, -0.5) {$B$};
%		\node [style=species] (C) at (-1, 0.5) {$C$};
%		\node [style=species] (D) at (-1, -0.5) {$D$};
%           \node [style=transition] (a) at (-2.5, 0) {$\alpha$}; 

		\node [style = transition] (b) at (2.5, 1) {$\beta$};
		\node [style = transition] (c) at (2.5, -1) {$\gamma$};
		\node [style = species] (E) at (1, 0) {$E$};
		\node [style = species] (F) at (4,0) {$F$};
		
%		\node [style=empty] (X) at (-5.1, 1) {$X$};
%		\node [style=none] (Xtr) at (-4.75, 0.75) {};
%		\node [style=none] (Xbr) at (-4.75, -0.75) {};
%		\node [style=none] (Xtl) at (-5.4, 0.75) {};
%            \node [style=none] (Xbl) at (-5.4, -0.75) {};
	
%		\node [style=inputdot] (1) at (-5, 0.5) {};
%		\node [style=empty] at (-5.2, 0.5) {$1$};
%		\node [style=inputdot] (2) at (-5, 0) {};
%		\node [style=empty] at (-5.2, 0) {$2$};
%		\node [style=inputdot] (3) at (-5, -0.5) {};
%		\node [style=empty] at (-5.2, -0.5) {$3$};

		\node [style=empty] (Y) at (-0.1, 1) {$Y$};
		\node [style=none] (Ytr) at (.25, 0.75) {};
		\node [style=none] (Ytl) at (-.4, 0.75) {};
		\node [style=none] (Ybr) at (.25, -0.75) {};
		\node [style=none] (Ybl) at (-.4, -0.75) {};

		\node [style=inputdot] (4) at (0, 0.5) {};
		\node [style=empty] at (-0.2, 0.5) {$4$};
		\node [style=inputdot] (5) at (0, -0.5) {};
		\node [style=empty] at (-0.2, -0.5) {$5$};		
		
		\node [style=empty] (Z) at (5, 1) {$Z$};
		\node [style=none] (Ztr) at (4.75, 0.75) {};
		\node [style=none] (Ztl) at (5.4, 0.75) {};
		\node [style=none] (Zbl) at (5.4, -0.75) {};
		\node [style=none] (Zbr) at (4.75, -0.75) {};

		\node [style=inputdot] (6) at (5, 0) {};
		\node [style=empty] at (5.2, 0) {$6$};	
		
	\end{pgfonlayer}
	\begin{pgfonlayer}{edgelayer}
%		\draw [style=inarrow] (A) to (a);
%		\draw [style=inarrow] (B) to (a);
%		\draw [style=inarrow] (a) to (C);
%		\draw [style=inarrow] (a) to (D);
		\draw [style=inarrow, bend left=30, looseness=1.00] (E) to (b);
		\draw [style=inarrow, bend left=30, looseness=1.00] (b) to (F);
		\draw [style=inarrow, bend left=30, looseness=1.00] (c) to (E);
		\draw [style=inarrow, bend left=30, looseness=1.00] (F) to (c);
%		\draw [style=inputarrow] (1) to (A);
%		\draw [style=inputarrow] (2) to (B);
%		\draw [style=inputarrow] (3) to (B);
%		\draw [style=inputarrow] (4) to (C);
%		\draw [style=inputarrow] (5) to (D);
		\draw [style=inputarrow] (4) to (E);
		\draw [style=inputarrow] (5) to (E);
		\draw [style=inputarrow] (6) to (F);
%		\draw [style=simple] (Xtl.center) to (Xtr.center);
%		\draw [style=simple] (Xtr.center) to (Xbr.center);
%		\draw [style=simple] (Xbr.center) to (Xbl.center);
%		\draw [style=simple] (Xbl.center) to (Xtl.center);
		\draw [style=simple] (Ytl.center) to (Ytr.center);
		\draw [style=simple] (Ytr.center) to (Ybr.center);
		\draw [style=simple] (Ybr.center) to (Ybl.center);
		\draw [style=simple] (Ybl.center) to (Ytl.center);
		\draw [style=simple] (Ztl.center) to (Ztr.center);
		\draw [style=simple] (Ztr.center) to (Zbr.center);
		\draw [style=simple] (Zbr.center) to (Zbl.center);
		\draw [style=simple] (Zbl.center) to (Ztl.center);
	\end{pgfonlayer}
\end{tikzpicture}
\]
the first step in composing $P$ and $Q$ is to put the pictures together:
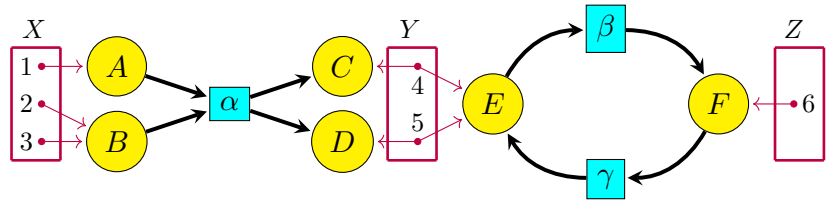
\[
\begin{tikzpicture}
	\begin{pgfonlayer}{nodelayer}
		\node [style=species] (A) at (-4, 0.5) {$A$};
		\node [style=species] (B) at (-4, -0.5) {$B$};
		\node [style=species] (C) at (-1, 0.5) {$C$};
		\node [style=species] (D) at (-1, -0.5) {$D$};
            \node [style=transition] (a) at (-2.5, 0) {$\alpha$}; 
		\node [style = transition] (b) at (2.5, 1) {$\beta$};
		\node [style = transition] (c) at (2.5, -1) {$\gamma$};
		\node [style = species] (E) at (1, 0) {$E$};
		\node [style = species] (F) at (4,0) {$F$};
		
		\node [style=empty] (X) at (-5.1, 1) {$X$};
		\node [style=none] (Xtr) at (-4.75, 0.75) {};
		\node [style=none] (Xbr) at (-4.75, -0.75) {};
		\node [style=none] (Xtl) at (-5.4, 0.75) {};
           \node [style=none] (Xbl) at (-5.4, -0.75) {};
	
		\node [style=inputdot] (1) at (-5, 0.5) {};
		\node [style=empty] at (-5.2, 0.5) {$1$};
		\node [style=inputdot] (2) at (-5, 0) {};
		\node [style=empty] at (-5.2, 0) {$2$};
		\node [style=inputdot] (3) at (-5, -0.5) {};
		\node [style=empty] at (-5.2, -0.5) {$3$};

		\node [style=empty] (Y) at (-0.1, 1) {$Y$};
		\node [style=none] (Ytr) at (.25, 0.75) {};
		\node [style=none] (Ytl) at (-.4, 0.75) {};
		\node [style=none] (Ybr) at (.25, -0.75) {};
		\node [style=none] (Ybl) at (-.4, -0.75) {};

		\node [style=inputdot] (4) at (0, 0.5) {};
		\node [style=empty] at (0, 0.25) {$4$};
		\node [style=inputdot] (5) at (0, -0.5) {};
		\node [style=empty] at (0, -0.25) {$5$};		
		
		\node [style=empty] (Z) at (5, 1) {$Z$};
		\node [style=none] (Ztr) at (4.75, 0.75) {};
		\node [style=none] (Ztl) at (5.4, 0.75) {};
		\node [style=none] (Zbl) at (5.4, -0.75) {};
		\node [style=none] (Zbr) at (4.75, -0.75) {};

		\node [style=inputdot] (6) at (5, 0) {};
		\node [style=empty] at (5.2, 0) {$6$};	
		
	\end{pgfonlayer}
	\begin{pgfonlayer}{edgelayer}
		\draw [style=inarrow] (A) to (a);
		\draw [style=inarrow] (B) to (a);
		\draw [style=inarrow] (a) to (C);
		\draw [style=inarrow] (a) to (D);
		\draw [style=inarrow, bend left=30, looseness=1.00] (E) to (b);
		\draw [style=inarrow, bend left=30, looseness=1.00] (b) to (F);
		\draw [style=inarrow, bend left=30, looseness=1.00] (c) to (E);
		\draw [style=inarrow, bend left=30, looseness=1.00] (F) to (c);
		\draw [style=inputarrow] (1) to (A);
		\draw [style=inputarrow] (2) to (B);
		\draw [style=inputarrow] (3) to (B);
		\draw [style=inputarrow] (4) to (C);
		\draw [style=inputarrow] (5) to (D);
		\draw [style=inputarrow] (4) to (E);
		\draw [style=inputarrow] (5) to (E);
		\draw [style=inputarrow] (6) to (F);
		\draw [style=simple] (Xtl.center) to (Xtr.center);
		\draw [style=simple] (Xtr.center) to (Xbr.center);
		\draw [style=simple] (Xbr.center) to (Xbl.center);
		\draw [style=simple] (Xbl.center) to (Xtl.center);
		\draw [style=simple] (Ytl.center) to (Ytr.center);
		\draw [style=simple] (Ytr.center) to (Ybr.center);
		\draw [style=simple] (Ybr.center) to (Ybl.center);
		\draw [style=simple] (Ybl.center) to (Ytl.center);
		\draw [style=simple] (Ztl.center) to (Ztr.center);
		\draw [style=simple] (Ztr.center) to (Zbr.center);
		\draw [style=simple] (Zbr.center) to (Zbl.center);
		\draw [style=simple] (Zbl.center) to (Ztl.center);
	\end{pgfonlayer}
\end{tikzpicture}
\]
At this point, if we ignore the sets $X,Y,Z$, we have a new Petri net whose set of places is the disjoint union of those for $P$ and $Q$.  The second step is to identify a place of $P$ with a place of $Q$ whenever both are images of the same point in $Y$.  We can then stop drawing everything involving $Y$, and get an open Petri net $Q \odot P \maps X \nrightarrow Z$:
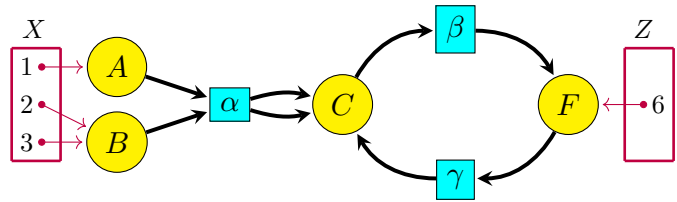
\[
\begin{tikzpicture}
	\begin{pgfonlayer}{nodelayer}
		\node [style=species] (A) at (-4, 0.5) {$A$};
		\node [style=species] (B) at (-4, -0.5) {$B$};;
             \node [style=transition] (a) at (-2.5, 0) {$\alpha$}; 
		\node [style = species] (E) at (-1, 0) {$C$};
		\node [style = species] (F) at (2,0) {$F$};

	     \node [style = transition] (b) at (.5, 1) {$\beta$};
		\node [style = transition] (c) at (.5, -1) {$\gamma$};
		
		\node [style=empty] (X) at (-5.1, 1) {$X$};
		\node [style=none] (Xtr) at (-4.75, 0.75) {};
		\node [style=none] (Xbr) at (-4.75, -0.75) {};
		\node [style=none] (Xtl) at (-5.4, 0.75) {};
             \node [style=none] (Xbl) at (-5.4, -0.75) {};
	
		\node [style=inputdot] (1) at (-5, 0.5) {};
		\node [style=empty] at (-5.2, 0.5) {$1$};
		\node [style=inputdot] (2) at (-5, 0) {};
		\node [style=empty] at (-5.2, 0) {$2$};
		\node [style=inputdot] (3) at (-5, -0.5) {};
		\node [style=empty] at (-5.2, -0.5) {$3$};	
		
		\node [style=empty] (Z) at (3, 1) {$Z$};
		\node [style=none] (Ztr) at (2.75, 0.75) {};
		\node [style=none] (Ztl) at (3.4, 0.75) {};
		\node [style=none] (Zbl) at (3.4, -0.75) {};
		\node [style=none] (Zbr) at (2.75, -0.75) {};

		\node [style=inputdot] (6) at (3, 0) {};
		\node [style=empty] at (3.2, 0) {$6$};	
		
	\end{pgfonlayer}
	\begin{pgfonlayer}{edgelayer}
		\draw [style=inarrow] (A) to (a);
		\draw [style=inarrow] (B) to (a);
	     \draw [style=inarrow, bend right=15, looseness=1.00] (a) to (E);
	     \draw [style=inarrow, bend left =15, looseness=1.00] (a) to (E);	
	     	\draw [style=inarrow, bend left=30, looseness=1.00] (E) to (b);
		\draw [style=inarrow, bend left=30, looseness=1.00] (b) to (F);
		\draw [style=inarrow, bend left=30, looseness=1.00] (c) to (E);
		\draw [style=inarrow, bend left=30, looseness=1.00] (F) to (c);	
		\draw [style=inputarrow] (1) to (A);
		\draw [style=inputarrow] (2) to (B);
		\draw [style=inputarrow] (3) to (B);
		\draw [style=inputarrow] (6) to (F);
		\draw [style=simple] (Xtl.center) to (Xtr.center);
		\draw [style=simple] (Xtr.center) to (Xbr.center);
		\draw [style=simple] (Xbr.center) to (Xbl.center);
		\draw [style=simple] (Xbl.center) to (Xtl.center);
		\draw [style=simple] (Ztl.center) to (Ztr.center);
		\draw [style=simple] (Ztr.center) to (Zbr.center);
		\draw [style=simple] (Zbr.center) to (Zbl.center);
		\draw [style=simple] (Zbl.center) to (Ztl.center);
	\end{pgfonlayer}
\end{tikzpicture}
\]

Formalizing this simple construction leads us into a bit of higher category theory.  The process of taking the disjoint union of two sets of places and then quotienting by an equivalence relation is a pushout.  Pushouts are defined only up to canonical isomorphism: for example, the place labeled $C$ in the last diagram above could equally well have been labeled $D$ or $E$.  This is why to get a category, with composition strictly associative, we need to use \emph{isomorphism classes} of open Petri nets as morphisms.  But there are advantages to working  with open Petri nets rather than isomorphism classes.  For example, we cannot point to a specific place or transition in an isomorphism class of Petri nets.  If we work with actual open Petri nets, we obtain not a category but a bicategory \cite{Congruence}.

However, this bicategory is equipped with more structure.   Besides composing open Petri nets, we can also `tensor' them via disjoint union: this describes Petri nets being run in parallel rather than in series.   The result is a symmetric monoidal bicategory.   Unfortunately, the axioms for a symmetric monoidal bicategory are cumbersome to check directly \cite{Stay}.  Double categories turn out to be much more convenient.  Double categories were introduced in the 1960s by Ehresmann \cite{Ehresmann63, Ehresmann65}.  More recently they have been used to study open dynamical systems \cite{Lerman,LS,N}, open electrical circuits and chemical reaction networks \cite{Bicategory}, open discrete-time Markov chains \cite{Panan}, coarse-graining for open continuous-time Markov chains \cite{BaezCourser}, and `tile logic' for concurrency in computer science \cite{Tile}.  

A 2-morphism in a double category can be drawn as a square:
\[
\begin{tikzpicture}[scale=1]
\node (D) at (-4,0.5) {$X_1$};
\node (E) at (-2,0.5) {$Y_1$};
\node (F) at (-4,-1) {$X_2$};
\node (A) at (-2,-1) {$Y_2.$};
\node (B) at (-3,-0.25) {$\Downarrow \alpha$};
\path[->,font=\scriptsize,>=angle 90]
(D) edge node [above]{$M$}(E)
(E) edge node [right]{$g$}(A)
(D) edge node [left]{$f$}(F)
(F) edge node [above]{$N$} (A);
\end{tikzpicture}
\]
We call $X_1,X_2,Y_1$ and $Y_2$ `objects', $f$ and $g$ `vertical 1-morphisms', $M$ and $N$ `horizontal 1-cells', and $\alpha$ a `2-morphism'.   We can compose vertical 1-morphisms to get new vertical 1-morphisms and compose horizontal 1-cells to get new horizontal 1-cells.  We can compose the 2-morphisms in two ways: horizontally and vertically.  This is just a quick sketch of the ideas; for full definitions see Appendix \ref{appendix}.

In Thm.\ \ref{thm:openpetri} we construct a symmetric monoidal 
double category $\Open(\Petri)$ with:
\begin{itemize}
\item  sets $X, Y, Z, \dots$ as objects, 
\item functions $f \maps X \to Y$ as vertical 1-morphisms,
\item open Petri nets $P \maps X \nrightarrow Y$ as horizontal 1-cells, 
\item morphisms between open Petri nets as 2-morphisms.
\end{itemize}
To get a feeling for morphisms between open Petri nets, some examples may be
helpful.   There is a morphism from this open Petri net:
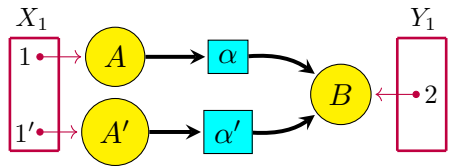
\[
\begin{tikzpicture}
	\begin{pgfonlayer}{nodelayer}

		\node [style = transition] (a) at (2.5, 0.5) {$\alpha$};
		\node [style = transition] (a') at (2.5, -0.5) {$\alpha'$};
		\node [style = species] (A) at (1, 0.5) {$A$};
		\node [style = species] (A') at (1, -0.5) {$A'$};
		\node [style = species] (B) at (4,0) {$B$};
	
		\node [style=empty] (Y) at (-0.1, 1) {$X_1$};
		\node [style=none] (Ytr) at (.25, 0.75) {};
		\node [style=none] (Ytl) at (-.4, 0.75) {};
		\node [style=none] (Ybr) at (.25, -0.75) {};
		\node [style=none] (Ybl) at (-.4, -0.75) {};

		\node [style=inputdot] (4) at (0, 0.5) {};
		\node [style=empty] at (-0.2, 0.5) {$1$};
		\node [style=inputdot] (5) at (0, -0.5) {};
		\node [style=empty] at (-0.2, -0.5) {$1'$};		
		
		\node [style=empty] (Z) at (5.1, 1) {$Y_1$};
		\node [style=none] (Ztr) at (4.75, 0.75) {};
		\node [style=none] (Ztl) at (5.4, 0.75) {};
		\node [style=none] (Zbl) at (5.4, -0.75) {};
		\node [style=none] (Zbr) at (4.75, -0.75) {};

		\node [style=inputdot] (6) at (5, 0) {};
		\node [style=empty] at (5.2, 0) {$2$};	
		
	\end{pgfonlayer}
	\begin{pgfonlayer}{edgelayer}

		\draw [style=inarrow] (A) to (a);
		\draw [style=inarrow, bend left=20, looseness=1.00] (a) to (B);
		\draw [style=inarrow] (A') to (a');
		\draw [style=inarrow, bend right=20, looseness=1.00] (a') to (B);
		\draw [style=inputarrow] (4) to (A);
		\draw [style=inputarrow] (5) to (A');
		\draw [style=inputarrow] (6) to (B);

		\draw [style=simple] (Ytl.center) to (Ytr.center);
		\draw [style=simple] (Ytr.center) to (Ybr.center);
		\draw [style=simple] (Ybr.center) to (Ybl.center);
		\draw [style=simple] (Ybl.center) to (Ytl.center);
		\draw [style=simple] (Ztl.center) to (Ztr.center);
		\draw [style=simple] (Ztr.center) to (Zbr.center);
		\draw [style=simple] (Zbr.center) to (Zbl.center);
		\draw [style=simple] (Zbl.center) to (Ztl.center);
	\end{pgfonlayer}
\end{tikzpicture}
\]
to this one:
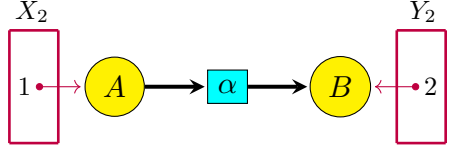
\[
\begin{tikzpicture}
	\begin{pgfonlayer}{nodelayer}

		\node [style = transition] (a) at (2.5, 0.0) {$\alpha$};
%		\node [style = transition] (a') at (2.5, -0.5) {$\alpha'$};
		\node [style = species] (A) at (1, 0.0) {$A$};
%		\node [style = species] (A') at (1, -0.5) {$A'$};
		\node [style = species] (B) at (4,0) {$B$};

		\node [style=empty] (Y) at (-0.1, 1) {$X_2$};
		\node [style=none] (Ytr) at (.25, 0.75) {};
		\node [style=none] (Ytl) at (-.4, 0.75) {};
		\node [style=none] (Ybr) at (.25, -0.75) {};
		\node [style=none] (Ybl) at (-.4, -0.75) {};

		\node [style=inputdot] (4) at (0, 0.0) {};
		\node [style=empty] at (-0.2, 0.0) {$1$};
%		\node [style=inputdot] (5) at (0, -0.5) {};
%		\node [style=empty] at (-0.2, -0.5) {$1'$};		
		
		\node [style=empty] (Z) at (5.1, 1) {$Y_2$};
		\node [style=none] (Ztr) at (4.75, 0.75) {};
		\node [style=none] (Ztl) at (5.4, 0.75) {};
		\node [style=none] (Zbl) at (5.4, -0.75) {};
		\node [style=none] (Zbr) at (4.75, -0.75) {};

		\node [style=inputdot] (6) at (5, 0) {};
		\node [style=empty] at (5.2, 0) {$2$};	
		
	\end{pgfonlayer}
	\begin{pgfonlayer}{edgelayer}
		\draw [style=inarrow] (A) to (a);
		\draw [style=inarrow] (a) to (B);

		\draw [style=inputarrow] (4) to (A);
%		\draw [style=inputarrow] (5) to (A');
		\draw [style=inputarrow] (6) to (B);

		\draw [style=simple] (Ytl.center) to (Ytr.center);
		\draw [style=simple] (Ytr.center) to (Ybr.center);
		\draw [style=simple] (Ybr.center) to (Ybl.center);
		\draw [style=simple] (Ybl.center) to (Ytl.center);
		\draw [style=simple] (Ztl.center) to (Ztr.center);
		\draw [style=simple] (Ztr.center) to (Zbr.center);
		\draw [style=simple] (Zbr.center) to (Zbl.center);
		\draw [style=simple] (Zbl.center) to (Ztl.center);
	\end{pgfonlayer}
\end{tikzpicture}
\]
mapping both primed and unprimed symbols to unprimed ones.  This describes a process of `simplifying' an open Petri net.  There are also morphisms that include  simple open Petri nets into more complicated ones.  For example, the above morphism of open Petri nets has a right inverse.

The main goal of this paper is to describe two forms of  semantics for open Petri nets.  The first is an `operational' semantics.  In Thm.\ \ref{thm:functoriality} we show this semantics gives a map from $\Open(\Petri)$ to a double category $\Open(\CMC)$.  This map sends any Petri net $P$ to the symmetric monoidal category $FP$, but it also acts on open Petri nets in a compositional way.   The second is a `reachability' semantics.  This gives a map from $\Open(\Petri)$ to the double category of relations, $\dRel$, which has:
\begin{itemize}
	\item sets $X, Y, Z, \dots$ as objects,
	\item functions $f \maps X \to Y$ as vertical 1-morphisms,
	\item relations $R \subseteq X \times Y$  as horizontal 1-cells,
	\item squares
	\[
	\begin{tikzpicture}[scale=1.5]
	\node (D) at (-4,0.5) {$X_1$};
	\node (E) at (-2,0.5) {$Y_1$};
	\node (F) at (-4,-1) {$X_2$};
	\node (A) at (-2,-1) {$Y_2$};
	\node (B) at (-3,-0.25) {};
	\path[->,font=\scriptsize,>=angle 90]
	(D) edge node [above]{$R \subseteq X_1 \times Y_1$}(E)
	(E) edge node [right]{$g$}(A)
	(D) edge node [left]{$f$}(F)
	(F) edge node [above]{$S \subseteq X_2 \times Y_2$} (A);
	\end{tikzpicture}
	\]
	obeying $(f \times g)R \subseteq S$ as 2-morphisms. 
\end{itemize}
In Petri net theory, a `marking' of a set $X$ is a finite multisubset of $X$: we can think of this
as a way of placing finitely many tokens on the points of $X$.   Let $\N[X]$  denote the set of markings of $X$.   Given an open Petri net $P \maps X \nrightarrow 
Y$, there is a `reachability relation' saying when a given marking of $X$ can be carried by a sequence of transitions in $P$ to a given marking of $Y$, leaving no tokens behind.    We write the reachability 
relation of $P$ as 
\[   \blacksquare P \subseteq \N[X] \times \N[Y].  \]  
In Thm.\ \ref{thm:reachability_1} we show that the map sending $P$ to $\blacksquare P$ extends to a lax double functor
\[    \blacksquare \colon \Open(\Petri) \to \dRel  .\]
In Thm.\ \ref{thm:reachability_2} we go further and show that this double functor is symmetric monoidal.

If the reader prefers bicategories to double categories, they may be relieved to learn that any double category $\lD$ gives rise to a bicategory $H(\lD)$ whose 2-morphisms are those 2-morphisms of $\lD$ of the form
\[
\begin{tikzpicture}[scale=1]
\node (D) at (-4,0.5) {$X$};
\node (E) at (-2,0.5) {$Y$};
\node (F) at (-4,-1) {$X$};
\node (A) at (-2,-1) {$Y.$};
\node (B) at (-3,-0.25) {$\Downarrow \alpha$};
\path[->,font=\scriptsize,>=angle 90]
(D) edge node [above]{$M$}(E)
(D) edge node [left]{$1_X$}(F)
(E) edge node [right]{$1_Y$}(A)
(F) edge node [above]{$N$} (A);
\end{tikzpicture}
\]
Shulman has described conditions under which symmetric monoidal double categories give rise to symmetric monoidal bicategories \cite{Shulman2}, and using his work one can show that the operational and reachability semantics for open Petri nets give maps between symmetric monoidal bicategories \cite{Structured}.   However, only the double category framework presents the operational and reachability semantics in their full glory.  Namely: using double categories, we can describe how these semantics behave on composite open Petri nets, tensor products of open Petri nets, and also morphisms between open Petri nets.

\section{From Petri Nets to Commutative Monoidal Categories}
\label{sec:presentation}

In this section we treat Petri nets as presentations of symmetric monoidal categories. As we shall explain, this has already been done by various authors.  Unfortunately there
are different notions of symmetric monoidal category, and also different notions of morphism between Petri nets, which combine to yield a confusing variety of possible approaches.

Here we take the maximally strict approach, and work with `commutative' monoidal categories.  This means we are treating tokens in Petri nets as indistinguishable rather than merely swappable---an approach known as the `collective token philosophy' \cite{GlabbeekPlotkin}.   A commutative monoidal category is a commutative monoid object in $\Cat$, so its associator: 
\[   \alpha_{a,b,c} \colon (a \otimes b) \otimes c \stackrel{\sim}{\longrightarrow} a \otimes (b \otimes c), \]
its left and right unitor:
\[    \lambda_a \maps I \otimes a \stackrel{\sim}{\longrightarrow} a ,  \qquad
       \rho_a \maps a \otimes I \stackrel{\sim}{\longrightarrow} a ,\]
and even---disturbingly---its symmetry:
\[    \sigma_{a,b} \maps a \otimes b \stackrel{\sim}{\longrightarrow} b \otimes a \]
are all identity morphisms.  The last would ordinarily be seen as `going too far', since while every  symmetric monoidal category is equivalent to one with trivial associator and unitors, this ceases to be true if we also require the symmetry to be trivial.  However, it seems that Petri nets most naturally serve to present symmetric monoidal categories of this very strict sort.  Thus, we construct a functor from the category of Petri nets to the category of commutative monoidal categories, which we call $\CMC$:
\[ F \colon \Petri \to \CMC .\]
This functor sends any Petri net $P$ to the free commutative monoidal category on $P$, and indeed it is a left adjoint.

It seems Montanari and Meseguer were the first to treat Petri nets as presentations of commutative monoidal categories \cite{MM}.   They constructed a closely related but different left adjoint functor from a category of Petri nets to a category of `Petri categories', which they call $\Cat\Petri$.  Our category $\Petri$ is a subcategory of their category of Petri nets: our morphisms of Petri nets send places to places, while they allow more general maps that send a place to a formal linear combination of places. On the other hand, their $\Cat\Petri$ is the full subcategory of $\CMC$ containing only commutative monoidal categories whose objects form a \emph{free} commutative monoid.

%revision: delete the following stuff:
%One reason we use our adjunction rather than theirs is that $\CMC$ has finite colimits, while $\Cat\Petri$ does not.  A commutative monoidal category with only identity morphisms is just a commutative monoid of objects.  A colimit of such commutative monoidal categories amounts to a colimit of commutative monoids.  Thus, using the fact that the full subcategory of \emph{free} commutative monoids does not have all coequalizers, or pushouts, we can see $\Cat\Petri$ also lacks these.   We need finite colimits to construct the symmetric monoidal double category $\Open(\CMC)$ in Thm.\ \ref{thm:openCMC}. 

In short, the situation is surprisingly subtle given the elementary nature of the concepts involved.   The paper by Montanari and Meseguer actually discusses over half a dozen categories of Petri nets and commutative monoidal categories.   Further work by Degano, Meseguer, Montanari \cite{DMM} and Sassone \cite{SassoneStrong, SassoneCategory, SassoneAxiomatization} explores other variations on the theme of generating symmetric monoidal categories from Petri nets.   Resisting the temptation to dwell on the subtleties of this topic, we present our approach with no further ado.

\begin{defn}
Let $\CMon$ be the category of commutative monoids and monoid homomorphisms.
\end{defn}	

\begin{defn}\label{N}
Let $J \maps \Set \to \CMon$ be the free commutative monoid functor, that is, the left adjoint of the functor $K \maps \CMon \to \Set$ that sends commutative monoids to their underlying sets and monoid homomorphisms to their underlying functions. Let 
\[ \N \maps \Set \to \Set \] be the free commutative monoid monad given by the composite $KJ$.
\end{defn}

For any set $X$, $\N[X]$ is the set of formal finite linear combinations of elements of $X$ with natural number coefficients.   The set $X$ naturally includes in $\N[X]$, and for any function $f \maps X \to Y$, $\N[f] \maps \N[X] \to \N[Y]$ is given by the unique monoid homomorphism that extends $f$.   

\begin{defn}\label{PetriNet}
We define a \define{Petri net} to be a pair of functions of the following form:
\[\xymatrix{ T \ar@<-.5ex>[r]_-t \ar@<.5ex>[r]^-s & \N[S]. } \]
We call $T$ the set of \define{transitions}, $S$ the set of \define{places}, $s$ the \define{source} function and $t$ the \define{target} function. 
\end{defn}

\begin{defn}\label{PetriMorphism}
A \define{Petri net morphism} from the Petri net $s,t\maps T \to \N[S]$ to the Petri net $s',t'\maps T \to \N[S']$ is a pair of functions $(f \maps T \to T', g \maps S \to S')$ such that the following diagrams commute:
	\[
	\xymatrix{ 
		T \ar[d]_f  \ar[r]^-{s} & \N[S] \ar[d]^-{\N[g]} \\	
		T' \ar[r]^-{s'} & \N[S'] 
	}
	\qquad
	\xymatrix{ 
		T \ar[d]_f  \ar[r]^-{t} & \N[S] \ar[d]^-{\N[g]} \\	
		T' \ar[r]^-{t'} & \N[S'] . 
	}
	\]
\end{defn}

\begin{defn}
Let $\Petri$ be the category of Petri nets and Petri net morphisms, with composition 
defined by 
\[  (f, g) \circ (f',g') = (f \circ f' , g \circ g')  .\]
\end{defn}

As mentioned above, Meseguer and Montanari \cite{MM} use a more general definition of Petri net morphism: they allow an arbitrary commutative monoid homomorphism from $\N[S]$ to $\N[S']$, not necessarily of the form $\N[g]$ for some function $g \maps S \to S'$.   Sassone \cite{SassoneStrong,SassoneCategory,SassoneAxiomatization} and Degano--Meseguer--Montanari \cite{DMM} also use this more general definition, but Baldan--Corradini--Ehrig--Heckel \cite{BCEH} and Baldan--Bonchi--Gadducci--Monreale \cite{BBGM} use the definition we are using here.  

\begin{defn}
	A \define{commutative monoidal category} is a commutative monoid object internal to $\Cat$. Explicitly, a commutative monoidal category is a strict monoidal category $(C,\otimes,I)$ such that for all objects $a$ and $b$ and morphisms $f$ and $g$ in $C$ 
	\[a \otimes b = b \otimes a \text{ and } f \otimes g = g \otimes f.\]
\end{defn}

Note that a commutative monoidal category is the same as a strict symmetric monoidal category where the symmetry isomorphisms $\sigma_{a,b} \maps a \otimes b \stackrel{\sim}{\longrightarrow} b \otimes a$ are all identity morphisms.   Every strict monoidal functor between commutative monoidal categories is automatically a strict symmetric monoidal functor.  This motivates the following definition:

\begin{defn}
	Let $\CMC$ be the category whose objects are commutative monoidal categories and whose morphisms are strict monoidal functors.
\end{defn}

We can turn a Petri net $P = (s,t \maps T \to \N[S])$ into a commutative monoidal category $FP$ as follows.  We take the commutative monoid of objects $\Ob(FP)$ to be the free commutative monoid on $S$.    We construct the commutative monoid of morphisms $\Mor(FP)$ as follows.  First we generate morphisms recursively:
\begin{itemize}
\item for every transition $\tau \in T$ we include a morphism $\tau \maps s(\tau) \to t(\tau)$;
\item for any object $a$ we include a morphism $1_a \maps a \to a$;
\item for any morphisms $f \maps a \to b$ and $g \maps a' \to b'$ we include a morphism denoted $f+g \maps a +a' \to b +b'$ to serve as their tensor product;
\item for any morphisms $f \maps a \to b$ and $g \maps b \to c$ we include a morphism $g\circ f \maps a \to c$ to serve as their composite.
\end{itemize}
Then we mod out by an equivalence relation on morphisms that imposes the laws of a commutative monoidal category, obtaining the commutative monoid $\Mor(FP)$.		
		
\begin{defn}
	Let $F \maps \Petri \to \CMC$ be the functor that makes the following assignments on Petri nets and morphisms:
	\[
	\xymatrix{ 
		T \ar[d]_f  \ar@<-.5ex>[r]_{t} \ar@<.5ex>[r]^{s} & \N[S] \ar[d]^{\N[g] \quad \mapsto} &  FP \ar[d]^{F(f,g)} \\
		T' \ar@<-.5ex>[r]_{t'} \ar@<.5ex>[r]^{s'} & \N[S'] & FP'.
	}
	\]
Here $F(f,g) \maps FP \to FP'$ is defined on objects by $\N [g]$. On morphisms, $F(f,g)$ is the unique map extending $f$ that preserves identities, composition, and the tensor product.
\end{defn}

\begin{lem}
\label{presentation}
The functor
\[ F \maps \Petri \to \CMC \] 
is a left adjoint.
\end{lem}

\begin{proof}
This is a special case of \cite[Thm.\ 5.1]{gen} which shows that there is similar adjunction for any Lawvere theory $\mathsf{Q}$. When $\mathsf{Q}$ is set equal to the Lawvere theory for commutative monoids this theorem gives the desired adjunction.
\end{proof}

\section{Open Petri Nets}
\label{sec:open_petri}

Our goal in this paper is to use the language of double categories to develop a theory of Petri nets with inputs and outputs that can be glued together.   The first step is to construct a double category 
$\Open(\Petri)$ whose horizontal 1-morphisms are open Petri nets.  For this we need a 
functor $L \maps \Set \to \Petri$ that maps any set $S$ to a Petri net with $S$ as its set of places, 
and we need $L$ to be a left adjoint.   

\begin{defn}
Let $L \maps \Set \to \Petri$ be the functor defined on sets and functions as follows:
\[
	\xymatrix{
		X \ar[d]_f^{ \quad \mapsto} & \emptyset \ar[d] \ar@<-.5ex>[r] \ar@<.5ex>[r] & \N[X] 		\ar[d]^{\N[f]} \\
		Y & \emptyset  \ar@<-.5ex>[r] \ar@<.5ex>[r] & \N[Y]
	}
\]
where the unlabeled maps are the unique maps of that type. 
\end{defn}

\begin{lem}
\label{L}
The functor $L$ has a right adjoint $R \maps \Petri \to \Set$ that acts as follows on 
Petri nets and Petri net morphisms:
\[ 	
	\xymatrix{ 
		T \ar[d]_f  \ar@<-.5ex>[r]_-{t} \ar@<.5ex>[r]^-{s} & \N[S] \ar[d]^{\N[g] \quad \mapsto} &  S \ar[d]^{g} \\
		T' \ar@<-.5ex>[r]_-{t'} \ar@<.5ex>[r]^-{s'} & \N[S] & S'.
	}
\]
\end{lem}

\begin{proof}
For any set $X$ and Petri net  $P = (s,t \maps T \to \N[S])$ we have natural isomorphisms
\[
\begin{array}{ccl}
\hom_{\Petri}\big(L(X), \xymatrix{ T \ar@<-.5ex>[r]_-t \ar@<.5ex>[r]^-s & \N[S] }\big)  
& \cong & \hom_{\Petri}\big(\xymatrix{ \emptyset \ar@<-.5ex>[r] \ar@<.5ex>[r] & \N[X] },
\xymatrix{ T \ar@<-.5ex>[r]_-t \ar@<.5ex>[r]^-s & \N[S] }\big) \\  
&\cong& \hom_{\Set}(X,S) \\ 
&\cong& \hom_{\Set}\big(X, R(\!\xymatrix{ T \ar@<-.5ex>[r]_-t \ar@<.5ex>[r]^-s & \N[S]}\!) \big).
\quad \qedhere \end{array} 
\]
\end{proof}

An `open' Petri net is a Petri net $P$ equipped with maps from two sets $X$ and $Y$ into its set of places, $RP$.  We can write this as a cospan in $\Set$ of the form
\[ \xymatrix{ & RP & \\
					X \ar[ur] & & Y. \ar[ul] }
\]	
Using the left adjoint $L$ we can reexpress this as a cospan in $\Petri$, and this gives our official definition:

\begin{defn}
\label{defn:openpetri}
An \define{open Petri net} is a diagram in $\Petri$ of the form
\[ \xymatrix{ & P & \\
					LX \ar[ur]^i & & LY \ar[ul]_o } 
\]
for some sets $X$ and $Y$.  We sometimes write this as $P \maps X \nrightarrow Y$ for short.
\end{defn}

We now introduce the main object of study: the double category $\Open(\Petri)$, which has open Petri nets as its horizontal 1-cells.  Since this is a symmetric monoidal double category, it involves quite a lot of structure. The definition of symmetric monoidal double category can be found in Appendix \ref{appendix}.
	
	\begin{thm}
	\label{thm:openpetri}
		There is a symmetric monoidal double category $\Open(\Petri)$ for which:
		\begin{itemize}
			\item objects are sets
			\item vertical 1-morphisms are functions
			\item horizontal 1-cells from a set $X$ to a set $Y$ are open Petri nets
				
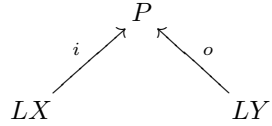
\[ \xymatrix{ & P & \\
					LX \ar[ur]^i & & LY \ar[ul]_o } \]	
			\item 2-morphisms $\alpha \maps P \Rightarrow P'$ are commutative diagrams
			  
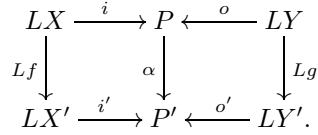
\[ \xymatrix{	LX \ar[r]^i \ar[d]_{Lf} & P \ar[d]_{\alpha} & LY \ar[l]_o \ar[d]^{Lg} \\
			  	LX' \ar[r]^{i'}  & P'  & LY'. \ar[l]_{o'} } \]
			  	in $\Petri$.
	\end{itemize}
Composition of vertical 1-morphisms is the usual composition of functions.   Composition of horizontal
1-cells is composition of cospans via pushout: given two horizontal 1-cells
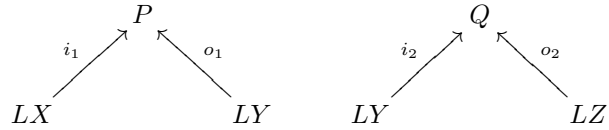
\[ \xymatrix{ & P & & & Q & \\
	LX \ar[ur]^{i_1} & & LY \ar[ul]_{o_1} & LY \ar[ur]^{i_2} & & LZ \ar[ul]_{o_2} }\]
their composite is given by this cospan from $LX$ to $LZ$:
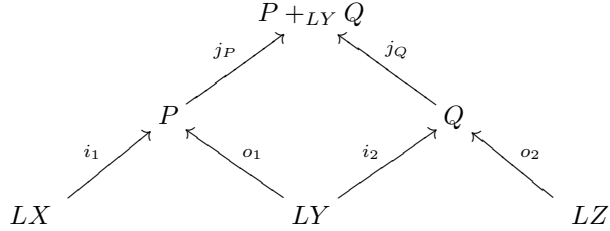
\[ \xymatrix{ 
	&   & P+_{LY} Q  &  & \\
	& P \ar[ur]^{j_P} &  & Q \ar[ul]_{j_Q} & \\
	LX \quad \ar[ur]^{i_1} & & LY \ar[ul]_{o_1}  \ar[ur]^{i_2} & & \quad LZ \ar[ul]_{o_2} }\]
where the diamond is a pushout square.  The horizontal composite of 2-morphisms
 
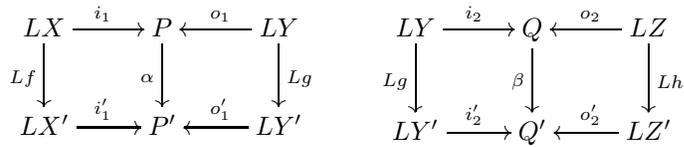
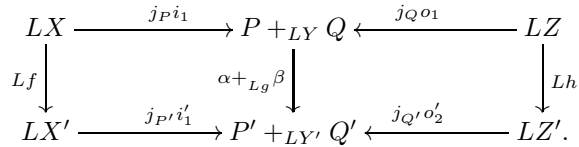
\[ \xymatrix{	LX \ar[r]^{i_1} \ar[d]_{Lf} & P \ar[d]_{\alpha} & LY \ar[l]_{o_1} \ar[d]^{Lg} \\
			  	LX' \ar[r]^{i'_1}  & P'  & LY' \ar[l]_{o'_1} } 
\qquad 
 \xymatrix{	LY \ar[r]^{i_2} \ar[d]_{Lg} & Q \ar[d]_{\beta} & LZ \ar[l]_{o_2} \ar[d]^{Lh} \\
			  	LY' \ar[r]^{i'_2}  & Q'  & LZ' \ar[l]_{o'_2} } \]
is given by
 \[ \xymatrix{	LX \ar[rr]^{j_P i_1} \ar[d]_{Lf} && P+_{LY} Q \ar[d]_{\alpha+_{{}_{Lg}} \beta} 
 && LZ \ar[ll]_{j_Q o_1} \ar[d]^{Lh} \\
LX' \ar[rr]^{j_{P'} i'_1}  && P'+_{LY'} Q'  && LZ'. \ar[ll]_{j_{Q'} o'_2} } \]
Vertical composition of 2-morphisms is done using composition of functions.
The symmetric monoidal structure comes from coproducts in $\Set$ and $\Petri$.
\end{thm}

\begin{proof} 
We construct this symmetric monoidal double category using the machinery of `structured cospans' \cite{Structured}.  The main tool is the following lemma, which explains the symmetric monoidal structure in more detail:

\begin{lem} \label{Courser}
Let $\A$ be a category with finite coproducts and $\X$ be a category with finite colimits. Given a left adjoint $L \maps \A \to \X$, there exists a unique symmetric monoidal double category $_L \Cospan(\X)$,  such that:
	\begin{itemize}
		\item objects are objects of $\A$,
		\item vertical 1-morphisms are morphisms of $\A$,
		\item a horizontal 1-cell from $a \in \A$ to $b\in \A$ is a cospan in $\X$ of this form:
			\[
			\begin{tikzpicture}[scale=1.5]
			\node (A) at (0,0) {$La$};
			\node (B) at (1,0) {$x$};
			\node (C) at (2,0) {$Lb$};
			\path[->,font=\scriptsize,>=angle 90]
			(A) edge node[above,left]{$$} (B)
			(C)edge node[above]{$$}(B);
			\end{tikzpicture}
			\]
		\item{a 2-morphism is a commutative diagram in $\X$ of this form:
			\[
			\begin{tikzpicture}[scale=1.5]
			\node (E) at (3,0) {$La$};
			\node (F) at (5,0) {$Lb$};
			\node (G) at (4,0) {$x$};
			\node (E') at (3,-1) {$Lc$};
			\node (F') at (5,-1) {$Ld$.};
			\node (G') at (4,-1) {$y$};
			\path[->,font=\scriptsize,>=angle 90]
			(F) edge node[above]{$$} (G)
			(E) edge node[left]{$Lf$} (E')
			(F) edge node[right]{$Lg$} (F')
			(G) edge node[left]{$h$} (G')
			(E) edge node[above]{$$} (G)
			(E') edge node[below]{$$} (G')
			(F') edge node[below]{$$} (G');
			\end{tikzpicture}
			\]
			}		
			\end{itemize}
Composition of vertical 1-morphisms is composition in $\A$.   Composition of horizontal
1-cells is composition of cospans in $\X$ via pushout: given horizontal 1-cells  
\[ \xymatrix{ & x & & & y & \\
	La \ar[ur]^{i_1} & & Lb \ar[ul]_{o_1} & Lb \ar[ur]^{i_2} & & Lc \ar[ul]_{o_2} }\]
their composite is this cospan from $La$ to $Lc$:
\[ \xymatrix{ 
	&   & x+_{Lb} y  &  & \\
	& x \ar[ur]^{j_x} &  & y \ar[ul]_{j_y} & \\
	La \quad \ar[ur]^{i_1} & & Lb \ar[ul]_{o_1}  \ar[ur]^{i_2} & & \quad Lc \ar[ul]_{o_2} }\]
where the diamond is a pushout square.  The horizontal composite of 2-morphisms
 \[ \xymatrix{	La \ar[r]^{i_1} \ar[d]_{Lf} & x \ar[d]_{\alpha} & Lb \ar[l]_{o_1} \ar[d]^{Lg} \\
			  	La' \ar[r]^{i'_1}  & x'  & Lb' \ar[l]_{o'_1} } 
\qquad 
 \xymatrix{	Lb \ar[r]^{i_2} \ar[d]_{Lg} & y \ar[d]_{\beta} & Lc \ar[l]_{o_2} \ar[d]^{Lh} \\
			  	Lb' \ar[r]^{i'_2}  & y'  & Lc' \ar[l]_{o'_2} } \]
is given by
 \[ \xymatrix{	La \ar[rr]^-{j_x i_1} \ar[d]_{Lf} && x+_{Lb} y \ar[d]_{\alpha+_{{}_{Lg}} \beta} 
 && Lc \ar[ll]_-{j_{y} o_2} \ar[d]^{Lh} \\
La' \ar[rr]^-{j_{x'} i'_1}  && x'+_{Lb'} y'  && Lc'. \ar[ll]_-{j_{y'} o'_2} } \]
The vertical composite of 2-morphisms
\[     \xymatrix{	La \ar[r]^{i_1} \ar[d]_{Lf} & x \ar[d]_{\alpha} & Lb \ar[l]_{o_1} \ar[d]^{Lg} \\
			  	La' \ar[r]^{i'_1}  & x'  & Lb' \ar[l]_{o'_1} } \]
\[   \xymatrix{	La' \ar[r]^{i'_1} \ar[d]_{Lf'} & x' \ar[d]_{\alpha'} &
 Lb' \ar[l]_{o'_1} \ar[d]^{Lg'} \\
			  	La'' \ar[r]^{i''_1}  & x''  & Lb'' \ar[l]_{o''_1} } \]
is given by
\[     \xymatrix{	La \ar[r]^{i_1} \ar[d]_{L(f'f)} & x \ar[d]_{\alpha'\alpha} & Lb \ar[l]_{o_1} \ar[d]^{L(g'g)} \\
			  	La'' \ar[r]^{i''_1}  & x''  & Lb''. \ar[l]_{o''_1} } \]
The tensor product is defined using 
chosen coproducts in $\A$ and $\X$.  Thus, the tensor product of two
objects $a_1$ and $a_2$ is $a_1+a_2$, the tensor product of two vertical 1-morphisms 
 \[   \xymatrix{	a_1 \ar[d]_{f_1} \\  b_1 } \qquad \qquad  \xymatrix{ a_2 \ar[d]_{f_2} \\ b_2 } \]
 is 
 \[ \xymatrix{	a_1+a_2 \ar[d]_{f_1+f_2} \\ b_1+b_2, } \]
the tensor product of two horizontal 1-cells 
\[ \xymatrix{ 	La_1  \ar[r]^{i_1} & x_1 & Lb_1 \ar[l]_{o_1} }   \qquad \qquad 
\xymatrix {	La_2  \ar[r]^{i_2} & x_2 & Lb_2 \ar[l]_{o_2} } \]
is
\[ \xymatrix{
	L(a_1 +a_2) \ar[rr]^{i_1+i_2} &&  x_1+ x_2 && L(b_1 + b_2), \ar[ll]_{o_1 + o_2} } \]
and the tensor product of two 2-morphisms 
 \[ \xymatrix{	La_1 \ar[r]^{i_1} \ar[d]_{Lf_1} & x_1 \ar[d]_{\alpha_1} & Lb_1 \ar[l]_{o_1} \ar[d]^{Lg_1} \\
			  	La'_1 \ar[r]^{i'_1}  & x'_1  & Lb'_1 \ar[l]_{o'_1} } \qquad
 \xymatrix{	La_2 \ar[r]^{i_2} \ar[d]_{Lf_2} & x_2 \ar[d]_{\alpha_2} & Lb_2 \ar[l]_{o_2} \ar[d]^{Lg_2} \\
			  	La'_2 \ar[r]^{i'_2}  & x'_2  & Lb'_2 \ar[l]_{o'_2} } \]
is
\[   \xymatrix{	L(a_1+a_2) \ar[rr]^{i_1+i_2} \ar[d]_{L(f_1+f_2)} && x_1+x_2 \ar[d]_{\alpha_1+\alpha_2} && L(b_1+b_2) \ar[ll]_{o_1 + o_2} \ar[d]^{L(g_1+g_2)} \\
			  	L(a'_1+a'_2) \ar[rr]^{i'_1+i'_2} && x'_1+x'_2  && L(b'_1+b'_2). \ar[ll]_{o'_1+o'_2} } \]
The units for these tensor products are taken to be initial objects, and the symmetry is defined using the canonical isomorphisms $a + b \cong b + a$.
\end{lem}

\begin{proof}  This is \cite[Thm.\ 3.9]{Structured}.  Note that we are abusing language slightly above.   We must choose a specific coproduct for each pair of
objects in $\X$ and $\A$ to give $_L \Cospan(X)$ its tensor product.   Given 
morphisms $i_1 \maps La_1 \to x_1$ and $i_2 \maps La_2 \to x_2$, their coproduct is really a morphism $i_1 + i_2 \maps La_1 + La_2 \to x_1 + x_2$ between
these chosen coproducts.  But since $L$ preserves coproducts, we can compose this 
morphism with the canonical isomorphism $L(a_1 + a_2) \cong La_1 + La_2$ to obtain the morphism that we call $i_1 + i_2 \maps L(a_1 + a_2) \to x_1 + x_2$ above.
\end{proof}

To apply this lemma to the situation at hand we need the following result.

\begin{lem}\label{Petricolimits}
$\Petri$ has small colimits.
\end{lem}

\begin{proof}
Note that $\Petri$ is equivalent to the comma category  $f/g$ where $f \maps \Set \to \Set$ is the identity and $g \maps \Set \to \Set$ is the functor $\N[-]^2$.  Whenever categories $\A$ and $\B$ have small colimits, $f \maps \A \to \C$ is a functor preserving such colimits, and $g \maps \B \to \C$ is any functor, then $f/g$ has small colimits \cite[Thm.\ 3, Sec.\ 5.2]{BR}.   Thus, $\Petri$ has small colimits.

For completeness, we recall how these colimits are constructed. The notation is simpler in the general case.  A diagram $D \maps \J \to f/g$ consists of diagrams $D_A \maps \J \to \A$ and $D_B \maps \J \to \B$ together with a natural transformation 
\[  \gamma \maps f \circ D_A \to g \circ D_B . \]
To construct the colimit of $D$, we use the canonical morphisms
\[   \alpha \maps \colim \, f \circ D_A \to f (\colim D_A), \] 
\[   \beta \maps \colim \, g \circ D_B \to g (\colim D_B) \]
defined using the universal property of the colimits at left. Since $f$ preserves colimits, $\alpha$ is an isomorphism.  We also use the fact that colimits are functorial, so that $\gamma$ gives a natural transformation that we may call
\[   \colim \, \gamma \colon 
\colim \, f \circ D_A \to \colim \, g \circ D_B .\]
The desired colimiting object $\colim D$ in $f/g$ consists of the objects $\colim D_A \in \A$, $\colim D_B \in \B$ and the morphism
\[  f(\colim D_A) \xrightarrow{\alpha^{-1}} 
\colim \, f \circ D_A 
\xrightarrow{\colim \,\gamma} 
\colim \, g \circ D_B 
\xrightarrow{\beta}
 g(\colim D_B) .\]
 
In particular, a diagram of Petri nets $D \maps \J \to \Petri$ gives rise to functors $D_A, D_B \maps \J \to \Set$, a Petri net
\[  \xymatrix{ D_A(j) \ar@<-.5ex>[r]_-{t_j} \ar@<.5ex>[r]^-{s_j} & \N[D_B(j)] }\]
for each object $j$ of $\J$, and a morphism between these Petri nets for each morphism of $\J$.  The colimit of $D$ takes the form
\[  \xymatrix{ \colim D_A \ar@<-.5ex>[r]_-{t} \ar@<.5ex>[r]^-{s} & \N[\colim D_B]. }\]
where $s$ and $t$ are constructed using the general prescription just described. 
\end{proof}

We now have all of the ingredients to apply Lemma \ref{Courser} to the functor $L \maps
\Set \to \Petri$.  Thm.\ \ref{thm:openpetri} follows from realizing that $\Open(\Petri)$ 
as described in the theorem is the symmetric monoidal double category 
$_L \Cospan(\Petri)$.   \end{proof}

\section{The Operational Semantics}

In Section \ref{sec:presentation} we saw how a Petri net $P$ gives a commutative monoidal category $FP$, and in Section \ref{sec:open_petri} we constructed a double category $\Open(\Petri)$ of open Petri nets.  Now we construct a double category $\Open(\CMC)$ of `open commutative monoidal categories' and a map 
\[             \Cospan(F) \maps \Open(\Petri) \to \Open(\CMC) .\]
This can be seen as providing an operational semantics for open Petri nets in which any open Petri net is mapped to the commutative monoidal category it presents.  The reachability semantics for open Petri nets is based on this more fundamental form of semantics.

The key is this commutative diagram of left adjoint functors:
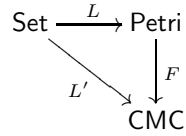
\[ \xymatrix{ \Set \ar[r]^L \ar[dr]_{L'} & \Petri \ar[d]^F\\
	& \CMC}
\] 
where $L' = FL$ sends any set to the free commutative monoidal category on this set: $L'X$ has $\N[X]$ as its set of objects, and only identity morphisms.  Using Lemma \ref{Courser},  we can produce two symmetric monoidal double categories from this diagram.   We have already seen one: $\Open(\Petri) = {_L\Cospan(\Petri)}$.  We now introduce the other:  $\Open(\CMC) = {_{L'}\Cospan(\CMC)}$.

\begin{thm}
\label{thm:openCMC}
There is a symmetric monoidal double category $\Open(\CMC)$ for which:
\begin{itemize}
		\item objects are sets
		\item vertical 1-morphisms are functions
			\item horizontal 1-cells from a set $X$ to a set $Y$ are \define{open commutative monoidal categories} $C \maps X \nrightarrow Y$, that is, cospans in $\CMC$ of the form
				
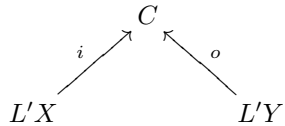
\[ \xymatrix{ & C & \\
					L'X \ar[ur]^i & & L'Y \ar[ul]_o } \]
		where $C$ is a commutative monoidal category and $i,o$ are strict monoidal functors,	
			\item 2-morphisms $\alpha \maps C \Rightarrow C'$ are commutative diagrams in $\CMC$
			of the form
			  \[ \xymatrix{	L'X \ar[r]^i \ar[d]_{L'f} & C \ar[d]_{\alpha} & L'Y \ar[l]_o \ar[d]^{L'g} \\
			  	L'X' \ar[r]^i  & C'  & L'Y'. \ar[l]_{o'} } \]
	\end{itemize}
and the rest of the structure is given as in Lemma \ref{Courser}.
\end{thm}

\begin{proof} 
To apply Lemma \ref{Courser} to the functor $L' \maps \Set \to \CMC$ we just need to check that $\CMC$ has finite colimits.  First note that
\[\CMC \simeq \mathsf{Mod}(\mathsf{CMON},\Cat)\]
where $\mathsf{Mod}(\mathsf{CMON},\Cat)$ is the category of finite product preserving functors from the Lawvere theory for commutative monoids to $\Cat$.  The cocompleteness of this category then follows from various classical results, some listed in the introduction of a paper by Freyd and Kelly \cite{FK}.  More recently, Trimble \cite[Prop.\ 3.1]{Trimble} showed that for any Lawvere theory $\Q$ and any cocomplete cartesian category $\X$ with finite products distributing over colimits, the category of finite-product-preserving functors $\mathsf{Mod}(\Q,\X)$ is cocomplete.
\end{proof}

The functor $F \maps \Petri \to \CMC$ induces a map sending open Petri nets to open commutative monoidal categories.  This map is actually part of a `symmetric monoidal double functor', a concept recalled in Appendix \ref{appendix}. 

\begin{thm}
\label{thm:functoriality}
There is a symmetric monoidal double functor 
\[   \Open(F) \maps \Open(\Petri) \to \Open(\CMC) \]
that is the identity on objects and vertical 1-morphisms, and makes the following assignments on horizontal 1-cells and 2-morphisms:
  \[ \xymatrix{	LX \ar[r]^i \ar[d]_{Lf} & P \ar[d]_{\alpha} & LY \ar[l]_o \ar[d]^{Lg \qquad {\Huge{\mapsto}} \quad } & &
  L'X \ar[r]^{Fi} \ar[d]_{L'f} & FP \ar[d]_{F\alpha} & L'Y \ar[l]_{Fo} \ar[d]^{L'g } 
   \\
			  	LX' \ar[r]^{i'} & P'  & LY' \ar[l]_{o'} & &
			  	L'X' \ar[r]^{Fi'}  & FP'  & L'Y'. \ar[l]_{Fo'} } \qedhere
\]
\end{thm}

\begin{proof}
This follows from the theory of structured cospans.  More generally, suppose $\A$ is a category with finite coproducts and $\X,\X'$ are categories with finite colimits.   Suppose there is a commuting triangle of left adjoints
\[ \xymatrix{ \A \ar[r]^L \ar[dr]_{L'} & \X \ar[d]^-F\\
	& \X'. }
\] 
Then Lemma \ref{Courser} gives us symmetric monoidal double categories ${}_L\Cospan(\X)$
and ${}_{L'}\Cospan(\X')$, and a result of the first author and Courser \cite[Thm.\ 4.3]{Structured} gives a symmetric monoidal double functor 
\[  \Cospan(F) \maps {}_L\Cospan(\X) \to {}_{L'}\Cospan(\X') \]
that is the identity on objects
and vertical morphisms, and acts as follows on horizontal 1-cells and 2-morphisms:
 \[ \xymatrix{	La \ar[r]^i \ar[d]_{Lf} & x \ar[d]_{\alpha} & Lb \ar[l]_o \ar[d]^{Lg \qquad {\Huge{\mapsto}} \quad } & &
  L'a \ar[r]^{Fi} \ar[d]_{L'f} & Fx \ar[d]_{F\alpha} & L'b \ar[l]_{Fo} \ar[d]^{L'g } 
   \\
			  	La' \ar[r]^{i'} & x'  & Lb' \ar[l]_{o'} & &
			  	L'a' \ar[r]^{Fi'}  & Fx'  & L'b'. \ar[l]_{Fo'} }
\]
In the case at hand, where the commutative triangle is
\[ \xymatrix{ \Set \ar[r]^L \ar[dr]_{L'} & \Petri \ar[d]^F\\
	& \CMC,}
\] 
this double functor $\Csp(F)$ is what we are calling $\Open(F)$.
\end{proof}

We can think of the commutative monoidal category $FP$ as providing an operational semantics for the Petri net $P$: morphisms in this category are processes allowed by the Petri net. The above theorem says that this semantics is compositional.  That is, if we write $P$ as a composite (or tensor product) of smaller open Petri nets, $FP$ will be the composite (or tensor product) of the corresponding open commutative monoidal categories. 

It is worthwhile comparing the work of some other authors.   Baldan, Corradini, Ehrig and Heckel \cite{BCEH} consider a category of Petri nets that is the same as our $\Petri$.   They define an `open net' to a Petri net $P$ equipped two subsets $X$ and $Y$ of its set of places. If one weakened this requirement slightly to demand merely that $X$ and $Y$ are equipped with injections into the set of places, the corresponding class of open Petri nets
\[ \xymatrix{ & P & \\
					LX \ar[ur]^i & & LY \ar[ul]_o }
\]
would be precisely those for which $i$ and $o$ are monic.  This class of open Petri nets is closed under our form of horizontal composition.  However, the authors take a different approach to composing open nets.  They consider a compositional semantics for open nets, but only for those of a special kind, called `deterministic occurrence nets' because there is never any choice about what a token can do.   They do not describe this semantics as a functor.

Bruni, Melgratti, Montanari and Soboci\'nski \cite{BMM,BMMS} also consider a category of Petri nets that matches our $\Petri$.  Given $m, n \in \N$, they define a `$P/T$-net with boundary' $P \maps m \to n$ to be a Petri net $P = \left( s, t \maps T \to \N[S]\right)$ equipped with maps $i \maps T \to \N^m, o \maps T \to \N^n$.   Thus, we may think of each transition as having, besides its usual source and target, an input which is a multisubset of $\{1, \dots, m\}$ and an output which is a multisubset of $\{1, \dots, n\}$.   They define a way to compose $P/T$-nets with boundary using `synchronization', and show this makes isomorphism classes of $P/T$-nets into the morphisms of a category.  They also describe an operational semantics for $P/T$ nets with boundary using a `tile calculus', which is essentially a double category \cite{Tile}.  However, the vertical direction in this double category has a fundamentally different meaning that in $\Open(\Petri)$: it is used to describe the process of firing transitions.  

As already mentioned, the operational semantics used here implements the `collective token philosophy', meaning that tokens are treated as indistinguishable.  By contrast, in the `individual token philosophy' swapping two tokens is treated as a nontrivial process.  Glabbeek and Plotkin argue that these philosophies give different interpretations of causality in Petri nets \cite{GlabbeekPlotkin}.  The key mathematical difference is that the individual token philosophy uses symmetric monoidal categories that are not commutative, so their symmetries are not identity morphisms.  Bruni \textit{et al.} showed that for a Petri net $P$, a category whose morphisms represent processes of $P$ under the individual token philosophy can be freely generated by equipping the inputs and outputs of each transition with an ordering \cite{BMeseguerMS}. Petri nets equipped with these orders are called `pre-nets'. In \cite[Sec.\ 6.1]{gen}, an operational semantics for pre-nets is described as a left adjoint
\[ Z \maps \mathsf{PreNet} \to \mathsf{SSMC}\]
where $\mathsf{PreNet}$ is an appropriate category of pre-nets and $\mathsf{SSMC}$ is the category of strict symmetric monoidal categories.  In a similar way to Thm.\ \ref{thm:functoriality}, this left adjoint can be extended to a symmetric monoidal double functor
\[ \Open(Z) \maps \Open(\mathsf{PreNet}) \to \Open(\mathsf{SSMC})\]
This double functor explicates the way in which the more nuanced semantics of the individual token philosophy can be built in a compositional way. A proof of existence and a detailed explanation of this double functor will be left to future work. 

\section{The Double Category of Relations}
\label{sec:relations}

Using the language of functorial semantics, $\Open(\Petri)$ can be thought of as a syntax for describing open systems, and reachability as a choice of semantics. To implement this, we show that the reachability relation of a Petri net can be defined for open Petri nets in a way that gives a lax double functor from $\Open(\Petri)$ to the double category of relations constructed by Grandis and  Par\'e \cite[Sec.\ 3.4]{GP1}.   Here we recall this double category and give it a symmetric monoidal structure.

This double category, which we call $\dRel$, has:
\begin{itemize}
	\item sets as objects,
	\item functions $f \maps X \to Y$ as vertical 1-morphisms from $X$ to $Y$,
	\item relations $R \subseteq X \times Y$  as horizontal 1-cells from $X$ to $Y$,
	\item squares
	\[
	\begin{tikzpicture}[scale=1.5]
	\node (D) at (-4,0.5) {$X_1$};
	\node (E) at (-2,0.5) {$Y_1$};
	\node (F) at (-4,-1) {$X_2$};
	\node (A) at (-2,-1) {$Y_2$};
	\node (B) at (-3,-0.25) {};
	\path[->,font=\scriptsize,>=angle 90]
	(D) edge node [above]{$R \subseteq X_1 \times Y_1$}(E)
	(E) edge node [right]{$g$}(A)
	(D) edge node [left]{$f$}(F)
	(F) edge node [above]{$S \subseteq X_2 \times Y_2$} (A);
	\end{tikzpicture}
	\]
	obeying $(f \times g)R \subseteq S$ as 2-morphisms. 
\end{itemize}
The last item deserves some explanation.  A preorder is a category such that for any pair of objects $a,b$ there exists at most one morphism $\alpha \maps x \to y$.  When such a morphism exists we usually write $x \le y$.  Similarly there is a kind of double category for which given any \define{frame}---that is, any collection of objects, vertical 1-morphisms and horizontal 1-cells as follows: 
\[
\begin{tikzpicture}[scale=1]
\node (D) at (-4,0.5) {$X_1$};
\node (E) at (-2,0.5) {$Y_1$};
\node (F) at (-4,-1) {$X_2$};
\node (A) at (-2,-1) {$Y_2$};
\node (B) at (-3,-0.25) {};
\path[->,font=\scriptsize,>=angle 90]
(D) edge node [above]{$M$}(E)
(E) edge node [right]{$g$}(A)
(D) edge node [left]{$f$}(F)
(F) edge node [above]{$N$} (A);
\end{tikzpicture}
\]
there exists at most one 2-morphism 
\[
\begin{tikzpicture}[scale=1]
\node (D) at (-4,0.5) {$X_1$};
\node (E) at (-2,0.5) {$Y_1$};
\node (F) at (-4,-1) {$X_2$};
\node (A) at (-2,-1) {$Y_2$};
\node (B) at (-3,-0.25) {$\Downarrow \alpha$};
\path[->,font=\scriptsize,>=angle 90]
(D) edge node [above]{$M$}(E)
(E) edge node [right]{$g$}(A)
(D) edge node [left]{$f$}(F)
(F) edge node [above]{$N$} (A);
\end{tikzpicture}
\]
filling this frame. Following \cite{BaezCourser} we call this a \define{degenerate} double category.  
Our definition of the 2-morphism in $\dRel$ will imply that this double category is degenerate.

Composition of vertical 1-morphisms in $\dRel$ is the usual composition of functions, while  
composition of horizontal 1-cells is the usual composition of relations.  Since composition of relations obeys the associative and unit laws strictly, $\dRel$ will be a \emph{strict} double category.   Since $\dRel$ is degenerate, there is at most one way to define the vertical composite of 2-morphisms 
\[
\begin{tikzpicture}[scale=1.5]
\node (D) at (-4,0) {$X_1$};
\node (E) at (-2,0) {$Y_1$};
\node (F) at (-4,-1.5) {$X_2$};
\node (A) at (-2,-1.5) {$Y_2$};
\node (B) at (-3,-0.75) {$\Downarrow \alpha$};
\node (C) at (-4,-3) {$X_3$};
\node (G) at (-2,-3) {$Y_3$};
\node (H) at (-3,-2.25) {$\Downarrow \beta$};
\node (I) at (-1,-1.5) {$=$};
\node (J) at (0,-0.5) {$X_1$};
\node (K) at (2,-0.5) {$Y_1$};
\node (L) at (0,-2.5) {$X_3$};
\node (M) at (2,-2.5) {$Y_3$};
\node (O) at (1,-1.5) {$\Downarrow \beta \alpha$};
\path[->,font=\scriptsize,>=angle 90]
(D) edge node [above]{$R \subseteq X_1 \times Y_1$}(E)
(E) edge node [right]{$g$}(A)
(D) edge node [left]{$f$}(F)
(F) edge node [left] {$f'$}(C)
(C) edge node [above] {$T \subseteq X_3 \times Y_3$} (G)
(A) edge node [right] {$g'$} (G)
(F) edge node [above]{$S \subseteq X_2 \times Y_2$} (A)
(J) edge node [above] {$R \subseteq X_1 \times Y_1$} (K)
(K) edge node [right] {$g' g$} (M)
(J) edge node [left] {$f' f$} (L)
(L) edge node [above] {$T \subseteq X_3 \times Y_3$} (M);
\end{tikzpicture}
\]
so we need merely check that a 2-morphism $\beta\alpha$ filling the frame at right exists.  This amounts to noting that
\[     (f \times g)R \subseteq S, \; (f' \times g')S \subseteq T \; \implies \;
(f' \times g')(f \times g)R \subseteq T .\]
Similarly, there is at most one way to define the horizontal composite of 2-morphisms
\[
\begin{tikzpicture}[scale=1.5]
\node (D) at (-3.5,0) {$X_1$};
\node (E) at (-2,0) {$Y_1$};
\node (F) at (-3.5,-1.5) {$X_2$};
\node (A) at (-2,-1.5) {$Y_2$};
\node (B) at (-2.75,-0.75) {$\Downarrow \alpha$};
\node (J) at (-0.5,0) {$Z_1$};
\node (L) at (-0.5,-1.5) {$Z_2$};
\node (O) at (-1.25,-0.75) {$\Downarrow \alpha'$};
\node (I) at (0.25,-0.75) {$=$};
\node (D') at (1,0) {$X_1$};
\node (E') at (2.5,0) {$Z_1$};
\node (F') at (1,-1.5) {$X_2$};
\node (A') at (2.5,-1.5) {$Z_2$};
\node (B') at (1.75,-0.75) {$\Downarrow \alpha' \circ \alpha$};
\path[->,font=\scriptsize,>=angle 90]
(D) edge node [above]{$R \subseteq X_1 \times Y_1$}(E)
(E) edge node [right]{$g$}(A)
(D) edge node [left]{$f$}(F)
(F) edge node [above]{$S \subseteq X_2 \times Y_2$} (A)
(E) edge node [above]{$R' \subseteq Y_1 \times Z_1$} (J)
(J) edge node [right]{$h$} (L)
(A) edge node [above]{$S' \subseteq Y_2 \times Z_2$} (L)
(D') edge node [above]{$R'R \subseteq X_1 \times Z_1$}(E')
(D') edge node [left]{$f$}(F')
(F') edge node [above]{$S'S \subseteq X_2 \times Z_2$} (A')
(E') edge node [right]{$h$} (A');
\end{tikzpicture}
\]
so we need merely check that a filler $\alpha' \circ \alpha$ exists, which amounts to noting that
\[   (f \times g)R \subseteq S , \; (g \times h)R' \subseteq S' \; \implies \;
(f \times h)(R'R) \subseteq S'S . \]

\begin{thm}
	There exists a strict double category $\dRel$ with the above properties.
\end{thm}

\begin{proof}
We use the definition of double category in Appendix \ref{appendix} (Def.\ \ref{defn:double_category}), which introduces two concepts not mentioned so far: the 
category of objects and the category of arrows.  We define the category of objects $\dRel_0$ to have sets as objects and functions as morphisms.   We define the category of arrows $\dRel_1$ to have relations as objects and squares 
	\[
	\begin{tikzpicture}[scale=1.5]
	\node (D) at (-4,0.5) {$X_1$};
	\node (E) at (-2,0.5) {$X_2$};
	\node (F) at (-4,-1) {$Y_1$};
	\node (A) at (-2,-1) {$Y_2$};
	\node (B) at (-3,-0.25) {};
	\path[->,font=\scriptsize,>=angle 90]
	(D) edge node [above]{$R \subseteq X_1 \times X_2$}(E)
	(E) edge node [right]{$g$}(A)
	(D) edge node [left]{$f$}(F)
	(F) edge node [above]{$S \subseteq Y_1 \times Y_2$} (A);
	\end{tikzpicture}
	\]
	with $(f \times g)R \subseteq S$ as morphisms.  The source and target functors $S,T \maps \dRel_1 \to \dRel_0$ are clear. The identity-assigning functor $u  \maps \dRel_0 \to \dRel_1$ sends a set $X$ to the identity function $1_X$ and a function $f \maps X \to Y$ to the unique 2-morphism
	\[
	\begin{tikzpicture}[scale=1.5]
	\node (D) at (-3.75,0.5) {$X$};
	\node (E) at (-2.25,0.5) {$X$};
	\node (F) at (-3.75,-1) {$Y$};
	\node (A) at (-2.25,-1) {$Y$};
	\node (B') at (-3,-0.25) {};
	\path[->,font=\scriptsize,>=angle 90]
	(D) edge node [above]{$1_X$}(E)
	(E) edge node [right]{$f$}(A)
	(D) edge node [left]{$f$}(F)
	(F) edge node [above]{$1_Y$} (A);
	\end{tikzpicture}
	\]
The composition functor $\odot \maps \dRel_1 \times_{\dRel_0} \dRel_1 \to \dRel_1$
acts on objects by the usual composition of relations, and it acts on 2-morphisms by horizontal composition as described above.  These functors can be shown to obey all the axioms of a double category.   In particular, because $\dRel$ is degenerate, all the required equations between 2-morphisms, such as the interchange law, hold automatically.
\end{proof}

Next we make $\dRel$ into a symmetric monoidal double category.  To do this, we first give $\dRel_0 = \Set$ the symmetric monoidal structure induced by the cartesian product. Then we give $\dRel_1$ a symmetric monoidal structure as follows.  Given relations $R_1 \subseteq X_1 \times Y_1$ and $R_2 \subseteq X_2 \times Y_2$, we define 
\[  R_1 \times R_2 = \{(x_1,x_2,y_1,y_2) : \; (x_1,y_1) \in R_1, (x_2,y_2) \in R_2 \} \subseteq X_1 \times X_2 \times Y_1 \times Y_2. \]
Given two 2-morphisms in $\dRel_1$:
\[
\begin{tikzpicture}[scale=1.5]
\node (D) at (-4,0.5) {$X_1$};
\node (E) at (-2,0.5) {$Y_1$};
\node (F) at (-4,-1) {$X_2$};
\node (A) at (-2,-1) {$Y_2$};
\node (D') at (-0.5,0.5) {$X'_1$};
\node (E') at (1.5,0.5) {$Y'_1$};
\node (F') at (-0.5,-1) {$X'_2$};
\node (A') at (1.5,-1) {$Y'_2$};
\node (B') at (0.5,-0.25) {$\Downarrow \alpha'$};
\node (B) at (-3,-0.25) {$\Downarrow \alpha$};
\path[->,font=\scriptsize,>=angle 90]
(D) edge node [above]{$R \subseteq X_1 \times Y_1$}(E)
(E) edge node [right]{$g$}(A)
(D) edge node [left]{$f$}(F)
(F) edge node [above]{$S \subseteq X_2 \times Y_2$} (A)
(D') edge node [above]{$R' \subseteq X'_1 \times Y'_1$}(E')
(E') edge node [right]{$g'$}(A')
(D') edge node [left]{$f'$}(F')
(F') edge node [above]{$S' \subseteq X'_2 \times Y'_2$} (A');
\end{tikzpicture}
\]
there is at most one way to define their product
\[
\begin{tikzpicture}[scale=1.5]
\node (D) at (-4,0.5) {$X_1 \times X'_1$};
\node (E) at (0,0.5) {$Y_1 \times Y'_1$};
\node (F) at (-4,-1) {$X_2 \times X'_2 $};
\node (A) at (0,-1) {$Y_2 \times Y'_2$};
\node (B') at (-2,-0.25) {$\Downarrow \alpha \times \alpha'$};
\path[->,font=\scriptsize,>=angle 90]
(D) edge node [above]{$R \times R' \subseteq (X_1 \times X_1')  \times (Y_1 \times Y'_1)$}(E)
(E) edge node [right]{$g \times g'$}(A)
(D) edge node [left]{$f \times f'$}(F)
(F) edge node [above]{$S \times S' \subseteq (X_2 \times X'_2) \times (Y_2 \times Y'_2) $} (A);
\end{tikzpicture}
\]
because $\dRel$ is degenerate.   To show that $\alpha \times \alpha'$ exists, we need
merely note that
\[    (f \times g) R \subseteq S, \; (f' \times g')R' \subseteq S' \; \implies \; 
(f \times f'\times g \times g') (R \times R') \subseteq S \times S'.  \]

\begin{thm}
The double category $\dRel$ can be given the structure of a symmetric monoidal double category with the above properties.
\end{thm}

\begin{proof}
We have described $\dRel_0$ and $\dRel_1$ as symmetric monoidal categories. The source and target functors $S,T \maps \dRel_1 \to \dRel_0$ are strict symmetric monoidal functors.   We must also equip $\dRel$ with two other pieces of structure.   One, called $\chi$, says how the composition of horizontal 1-cells interacts with the tensor product in the category of arrows.  The other, called $\mu$, says how the 
 identity-assigning functor $u$ relates the tensor product in the category of objects to the tensor product in the category of arrows.  These are defined as follows.  Given four horizontal 1-cells 
\[   R_1 \subseteq X_1 \times Y_1, \quad R_2 \subseteq Y_1 \times Z_1, \]
\[ S_1 \subseteq X_2 \times Y_2, \quad S_2 \subseteq Y_2 \times Z_2, \]
the globular 2-isomorphism $\chi \maps (R_2 \times S_2)(R_1 \times S_1) \Rightarrow (R_2 R_1) \times (S_2 S_1)$ is the identity 2-morphism
	\[
	\begin{tikzpicture}[scale=1.5]
	\node (D) at (-4,0.5) {$X_1 \times X_2$};
	\node (E) at (-1,0.5) {$Z_1 \times Z_2$};
	\node (F) at (-4,-1) {$X_1 \times X_2$};
	\node (A) at (-1,-1) {$Z_1 \times Z_2$};
	\node (B') at (-2.5,-0.25) {$$};
	\path[->,font=\scriptsize,>=angle 90]
	(D) edge node [above]{$(R_2 \times S_2)(R_1 \times S_1)$}(E)
	(E) edge node [right]{$1$}(A)
	(D) edge node [left]{$1$}(F)
	(F) edge node [above]{$(R_2 R_1) \times (S_2 S_1)$} (A);
	\end{tikzpicture}
	\]
	The globular 2-isomorphism $\mu \maps u(X \times Y) \Rightarrow u(X) \times u(Y)$ is
	the identity 2-morphism
	\[
	\begin{tikzpicture}[scale=1.5]
	\node (D) at (-4,0.5) {$X \times Y$};
	\node (E) at (-2,0.5) {$X \times Y$};
	\node (F) at (-4,-1) {$X \times Y$};
	\node (A) at (-2,-1) {$X \times Y$};
	\node (B') at (-3,-0.25) {$$};
	\path[->,font=\scriptsize,>=angle 90]
	(D) edge node [above]{$1_{X \times Y} $}(E)
	(E) edge node [right]{$1$}(A)
	(D) edge node [left]{$1$}(F)
	(F) edge node [above]{$1_X \times 1_Y$} (A);
	\end{tikzpicture}
	\]
All the commutative diagrams in the definition of symmetric monoidal double category (Defs.\ \ref{defn:monoidal_double_category} and \ref{defn:symmetric_monoidal_double_category}) 
can be checked straightforwardly.   In particular, all diagrams of
2-morphisms commute automatically because $\dRel$ is degenerate.  \end{proof}

\section{The Reachability Semantics}

Now we explain how $\Open(\Petri)$ provides a compositional approach to the reachability problem.  In particular, we prove that the reachability semantics defines a lax double functor 
\[   \blacksquare \maps \Open(\Petri) \to \dRel \]
which is symmetric monoidal.
 
\begin{defn} 
Let $P$ be a Petri net $(s,t \maps T \to \N[S])$.  A \define{marking} of $P$ is an element $m \in \N[S]$.  Given a transition $\tau \in T$, a \define{firing} of $\tau$ is a tuple $(\tau, m , n)$ such that $m \ge s(\tau)$ and $n + s(\tau) = m + t(\tau)$.  We say that a marking $n$ is \define{reachable} from a marking $m$ if for some $k \ge 1$ there is a sequence of markings $m = m_1 , \dots, m_k = n$ and firings $\{(\tau_i, m_i, m_{i+1})\}_{i=1}^{k-1}$.   In particular, taking $k = 1$, any marking is reachable from itself with no firings.
\end{defn}

Given two markings of a Petri net, the problem of deciding whether one is reachable from the other is called the `reachability problem'. In 1984 Mayr showed that the reachability problem is decidable \cite{Mayr}. However, it is a very hard problem: in 1976 Lipton had showed that it requires at least exponential space, and in fact any EXPSPACE algorithm can be reduced in polynomial time to a Petri net reachability problem \cite{Lipton}.   More recently, lower and upper bounds on the time to solve the reachability problem have been found \cite{Czerwinski,Leroux}.  The lower bound grows much faster than the Ackermann function.

There is a close connection between reachability and the free commutative monoidal category on a Petri net constructed in Lemma \ref{presentation}.

\begin{prop}\label{reachability}
If $m$ and $n$ are markings of a Petri net $P$, then $n$ is reachable from $m$ if and only if there is a morphism $f \maps m \to n$ in $FP$.
\end{prop}
 
 \begin{proof}
 If $n$ is reachable from $m$, there is a sequence of markings
 $m = m_1 , \dots, m_k = n$ and firings $\{(\tau_i, m_i, m_{i+1})\}_{i=1}^{k-1}$.
 For each firing $(\tau_i,m_i,m_{i+1})$ there is a morphism in $FP$ given by 
\[  \tau_i + 1_{m_i-s (\tau_i)} \maps m_i \to m_{i+1} . \]
Taking the composite of these morphisms gives a morphism $f \maps m \to n$ in $FP$. 

Conversely, if $f \maps m \to n$ is a morphism in $FP$, it can be obtained by composition
and addition (that is, the tensor product) from morphisms arising from the basic transitions and symmetry morphisms.  Because $+$ is a functor, we have the interchange law
\[ (f_1 \circ g_1) + (f_2 \circ g_2) = (f_1+f_2) \circ (g_1+g_2) \]
whenever $f_1,g_1$ and $f_2,g_2$ are pairs of composable morphisms in $FP$.  We can use 
this inductively to simplify $f$ into a composite of sums.   If $f_1 \maps a_1 \to b_1$ and $f_2 \maps a_2 \to b_2$ are morphisms in $FP$, the interchange law also tells us that
\[  f_1 + f_2 = (f_1 \circ 1_{a_1}) + (1_{b_2} \circ f_2) = (f_1 + 1_{b_2}) \circ (1_{a_1} +f_2).\]
This fact allows us to inductively simplify $f$ to a composite of sums each containing one transition. The factors in this composite correspond to firings that make $n$ reachable from $m$.   (Here we allow the possibility of an empty composite, which corresponds to an identity morphism.) \end{proof}
  
\begin{defn} We define the \define{reachability relation} of an open Petri net 
\[ 
  \xymatrix{
		 LX \ar[r]^{i} &P & LY \ar[l]_{o} 
		 }
\]
to be the relation 
\[ \blacksquare P = \{ (x,y) \in \N[X] \times \N[Y] |\ o(y) \text{ is reachable from } i(x)  \} \; \subseteq \; \N[X] \times \N[Y].\]
\end{defn}

Note that $\blacksquare P$ depends on the whole open Petri net $P \maps X \nrightarrow Y$, 
not just its underlying Petri net $P$.   By Prop.\ \ref{reachability}, 
\[ \blacksquare P = \{ (x,y) \in \N[X] \times \N[Y] |\ \exists h \maps F(i)(x) \to F(o)(y)\} .\]
Here $F(i)(x)$ and $F(o)(y)$ are objects of the category $FP$, and the reachability relation
holds iff there is a morphism in $FP$ from the first of these to the second.

\begin{thm}
\label{thm:reachability_1}
There is a lax double functor $\blacksquare \maps \Open(\Petri) \to \dRel$, called the \define{reachability semantics}, that sends
	\begin{itemize}
		\item any object $X$ to the underlying set of the free commutative monoid $\N[X]$,
		which we denote simply as $\N[X]$,
		\item any vertical 1-morphism $f \maps X \to Y$ to the underlying function of $\N[f]$,
		\item any horizontal 1-cell, that is, any open Petri net
		\[ \xymatrix{ 
					LX \ar[r]^{i} & P & LY, \ar[l]_{o} }  \] 
		to the reachability relation $\blacksquare P$.
		\item any 2-morphism $\alpha \maps P \Rightarrow P'$, that is any commuting diagram
		\[ \xymatrix{LX \ar[d]_{Lf} \ar[r]^{i} & P \ar[d]_{\alpha} & LY \ar[l]_o \ar[d]^{Lg} \\
					LX'  \ar[r]^{i'} & P' & LY', \ar[l]_{o'} } \]
	to the square
			\[
			\begin{tikzpicture}[scale=1.5]
			\node (D) at (-4,0.5) {$\N[X]$};
			\node (E) at (-1.5,0.5) {$\N[Y]$};
			\node (F) at (-4,-0.5) {$\N[X']$};
			\node (A) at (-1.5,-0.5) {$\N[Y'].$};
			\node (B) at (-3,-0.25) {$ $};
			\path[->,font=\scriptsize,>=angle 90]
			(D) edge node [above]{$\blacksquare P \subseteq X \times Y$}(E)
			(E) edge node [right]{$\N[g]$}(A)
			(D) edge node [left]{$\N[f]$}(F)
			(F) edge node [above]{$\blacksquare P' \subseteq X' \times Y'$} (A);
			\end{tikzpicture}
			\]
	\end{itemize}
\end{thm}

\begin{proof}
We construct $\blacksquare$ as the composite $G \circ \Cospan(F)$ where 
\[   \Cospan(F) \maps \Open(\Petri) \to \Open(\CMC) \]
is the double functor constructed in Thm.\ \ref{thm:functoriality} and 
\[    G \maps \Open(\CMC) \to \dRel  \]
is defined as follows.   Recall that we have categories of objects
\[    \Open(\CMC)_0 = \dRel_0 = \Set  .\]
We define $G_0 \maps \Open(\Petri\Cat)_0 \to \dRel_0$ to be the functor  $\N \maps \Set \to \Set$. We define $G_1 \maps \Open(\CMC)_1 \to \dRel_1$ as follows:
\[
\begin{tikzpicture}

\node (D) at (2.5,1) {$\N[X]$};
\node (E) at (6.2,1) {$\N[Y]$};
\node (F) at (2.5,-1) {$\N[X']$};
\node (A) at (6.2,-1) {$\N[Y'].$};
\node (B) at (4.5,0) {$ $};
\path[->,font=\scriptsize,>=angle 90]
(D) edge node [above]{$G_1 C \subseteq \N[X] \times \N[Y]$}(E)
(E) edge node [right]{$\N[g]$}(A)
(D) edge node [left]{$\N[f]$}(F)
(F) edge node [above]{$G_1 C' \subseteq \N[X'] \times \N[Y']$} (A);
\node at (1,0) {$\mapsto$};
\node (L'X) at (-4.5,1) {$L'X$};
\node (L'X') at (-4.5,-1) {$L'X'$};
\node (C') at (-2.5,-1) {$C'$};
\node (L'Y') at (-0.5,-1) {$L'Y'$};
\node (C) at (-2.5,1) {$C$};
\node (L'Y) at (-0.5,1) {$L'Y$};
\path[->,font=\scriptsize,>=angle 90]
(L'X) edge node[above]{$i$} (C)
(L'Y) edge node[above]{$o$} (C)
(L'X') edge node[above]{$i'$} (C')
(L'Y') edge node[above]{$o'$} (C')
(L'X) edge node[left]{$L'f$} (L'X')
(L'Y) edge node[right]{$L'g$} (L'Y')
(C) edge node[right]{$\alpha$} (C');
\end{tikzpicture}
\]
Recall that the set of objects of $L'X$ is $\N[X]$ and the set of objects of $L'Y$ is $\N[Y]$.
We define $G_1C$ to be the relation 
\[  \{ (x,y) \in L'X \times L'Y \; | \;  h \maps i(x) \to o(y) \textrm{ for some } h \textrm{ in } C \} \; \subseteq \; \N[X] \times \N[Y] \]
and $G_1 \alpha$ to be the inclusion 
\[   (\N[f] \times \N[g]) G_1 C \subseteq G_1 C' .\]  
To see that this inclusion is well-defined, suppose $(x,y) \in G_1C$. Then there exists a morphism $h \maps i(x) \to o(y)$ in $C$.   We thus have a morphism $\alpha(h) \maps \alpha(i(x)) \to \alpha(o(y))$ in $C'$.  However, on objects we have $\alpha \circ i = i' \circ L'f = i' \circ \N[f]$ and similarly $\alpha \circ o =  o' \circ \N[g]$, so $\alpha(h) \maps i'(\N[f] (x)) \to o'(\N[g] (y))$.   It follows that $(\N[f] \times \N[g])(x,y) \in G_1 C'$.
	
Next we prove that $G$ is a lax double functor.   First note that by construction we have the following equalities:
\[ S \circ G_1 = G_0 \circ S , \qquad T \circ G_1 = G_0 \circ T .\]
Next we need the composition comparison required by Def.\ \ref{defn:double_functor}.   Suppose we compose $C \maps X \nrightarrow Y$ and $D \maps Y \nrightarrow Z$ in $\Open(\CMC)$:
	\[ \xymatrix{ 
		&   & C+_{L'Y} D &  & \\
		& C \ar[ur]^-{j_C} &  & D\ar[ul]_-{j_D} & \\
		L'X  \phantom{X} \ar[ur]^-{i_1} & & L'Y \ar[ul]_-{o_1}  \ar[ur]^{i_2} & & {\phantom{X}} L'Z. \ar[ul]_{o_2} }
	\]
We need to prove that
\[      G_1(D) \odot G_1(C) \subseteq G_1(D \odot C) .\]
We have 
\[ G_1(D \odot C) = \{ (x,z) \in L'X \times L'Z \; | \; \exists h \maps j_C i_1 (x) \to j_D o_2(z) \}. \]
On the other hand, 
\[ G_1 C = \{ (x,y) \in L'X \times L'Y \;|\; \exists m \maps i_1(x) \to o_1(y) \}\] 
and 
\[ G_1 D = \{ (y,z) \in L'Y \times L'Z \;|\; \exists n \maps i_2(y) \to o_2(z)  \}\] 
which compose to give the relation
\[G_1 D \, \odot\, G_1 C = \{ (x,z) \in L'X \times L'Z \;|\; \exists y \;\, (x, y) \in G_1 C  
\textrm{ and } (y,z) \in G_1 D \}. \]
Suppose $(x,z) \in G_1 D \odot G_1 C$.   Then there exist morphisms $m \maps i_1(x) \to o_1(y)$ in $C$ and $n \maps i_2(y) \to o_2(z)$ in $D$.  By commutativity of the pushout square,
$j_C o_1 = j_D i_2$.  Therefore, the codomain of $j_C(m) $ is $j_C o_1(y) = j_D i_2(y)$, which is
also the domain of $j_D(n)$.  This allows us to form the composite 
\[   j_D(n) \circ j_C(m) \maps j_C i_1(x) \to j_D o_2(z) . \]
Thus $(x,z) \in G_1 ( D \odot C)$ as desired.

We also need the identity comparison required by Def.\ \ref{defn:double_functor}.   Thus, we need
\[      U_{G_0(X)} \subseteq G_1(U_X) \]
for any set $X$.    By definition, $U_X \in \Open(\CMC)_1$ is the cospan
\[   \xymatrix{L'X \ar[r]^{1} & L'X & \ar[l]_{1} L'X. } \]
Because $L'X$ has no non-identity morphisms, $G_1$ maps this to the identity relation on the set $\N[X]$.   On the other hand, $G_0(X) = \N[X]$ and $U_{G_0(X)}$ is the identity relation on this set.  So, the desired inclusion is actually an equality.

Finally, because $\dRel$ is a degenerate double category, the composition and identity comparisons for $G$ are trivially natural transformations.  For the same reason, the diagrams in Def.\ \ref{defn:double_functor} expressing compatibility with the associator, left unitor, and right unitor also commute trivially. 
It follows that $G$ is a lax double functor.

To complete the proof, one simply computes the composite $\blacksquare = G \circ \Cospan(F)$ and
checks that it matches the description in the theorem statement.
\end{proof}

The reachability semantics is only lax: given two open Petri nets $P \maps X \nrightarrow Y$ and $Q \maps Y \nrightarrow Z$, the composite of $\blacksquare Q$ and $\blacksquare P$ is in general a proper subset of $\blacksquare(Q \odot P)$.  To see this, take $P$ to be this open Petri net:
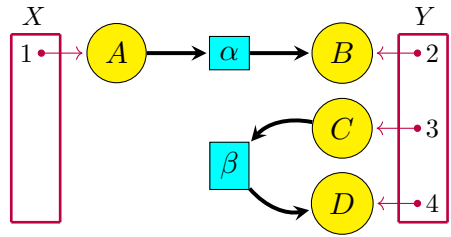
\[
\begin{tikzpicture}
	\begin{pgfonlayer}{nodelayer}
		\node [style=species] (A) at (-4, 1.5) {$A$};
		\node [style=species] (B) at (-1, 1.5) {$B$};
		\node [style=species] (C) at (-1, 0.5) {$C$};
		\node [style=species] (D) at (-1, -0.5) {$D$};
            \node [style=transition] (a) at (-2.5, 1.5) {$\alpha$}; 
             \node [style=transition] (b) at (-2.5, 0) {$\beta$}; 
		
		\node [style=empty] (X) at (-5.1, 2) {$X$};
		\node [style=none] (Xtr) at (-4.75, 1.75) {};
		\node [style=none] (Xbr) at (-4.75, -0.75) {};
		\node [style=none] (Xtl) at (-5.4, 1.75) {};
             \node [style=none] (Xbl) at (-5.4, -0.75) {};
	
		\node [style=inputdot] (1) at (-5, 1.5) {};
		\node [style=empty] at (-5.2, 1.5) {$1$};

		\node [style=empty] (Y) at (0.1, 2) {$Y$};
		\node [style=none] (Ytr) at (.4, 1.75) {};
		\node [style=none] (Ytl) at (-.25, 1.75) {};
		\node [style=none] (Ybr) at (.4, -0.75) {};
		\node [style=none] (Ybl) at (-.25, -0.75) {};

		\node [style=inputdot] (2) at (0, 1.5) {};
		\node [style=empty] at (0.2, 1.5) {$2$};
		\node [style=inputdot] (3) at (0, 0.5) {};
		\node [style=empty] at (0.2, 0.5) {$3$};
		\node [style=inputdot] (4) at (0, -0.5) {};
		\node [style=empty] at (0.2, -0.5) {$4$};		
		
%		\node [style=empty] (Z) at (3, 1) {$Z$};
%		\node [style=none] (Ztr) at (3.25, 0.75) {};
%		\node [style=none] (Ztl) at (2.75, 0.75) {};
%		\node [style=none] (Zbl) at (2.75, -0.75) {};
%		\node [style=none] (Zbr) at (3.25, -0.75) {};
		
	\end{pgfonlayer}
	\begin{pgfonlayer}{edgelayer}
		\draw [style=inarrow] (A) to (a);
		\draw [style=inarrow] (a) to (B);
		\draw [style=inarrow,bend right=30, looseness=1.00] (C) to (b);
		\draw [style=inarrow, bend right=30, looseness=1.00] (b) to (D);
		\draw [style=inputarrow] (1) to (A);
		\draw [style=inputarrow] (2) to (B);
		\draw [style=inputarrow] (3) to (C);
		\draw [style=inputarrow] (4) to (D);
	
		\draw [style=simple] (Xtl.center) to (Xtr.center);
		\draw [style=simple] (Xtr.center) to (Xbr.center);
		\draw [style=simple] (Xbr.center) to (Xbl.center);
		\draw [style=simple] (Xbl.center) to (Xtl.center);
		\draw [style=simple] (Ytl.center) to (Ytr.center);
		\draw [style=simple] (Ytr.center) to (Ybr.center);
		\draw [style=simple] (Ybr.center) to (Ybl.center);
		\draw [style=simple] (Ybl.center) to (Ytl.center);
	\end{pgfonlayer}
\end{tikzpicture}
\]
and take $Q$ to be this:
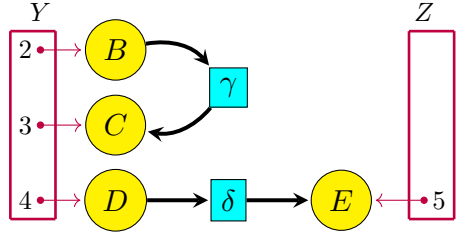
\[
\begin{tikzpicture}
	\begin{pgfonlayer}{nodelayer}
		\node [style=species] (B) at (-1, 1.5) {$B$};
		\node [style=species] (C) at (-1, 0.5) {$C$};
		\node [style=species] (D) at (-1, -0.5) {$D$};
		\node [style=species] (E) at (2, -0.5) {$E$};
            \node [style=transition] (c) at (0.5, 1) {$\gamma$}; 
             \node [style=transition] (d) at (0.5, -0.5) {$\delta$}; 
		
		\node [style=empty] (Z) at (3.1, 2) {$Z$};
		\node [style=none] (Ztr) at (2.9, 1.75) {};
		\node [style=none] (Zbr) at (2.9, -0.75) {};
		\node [style=none] (Ztl) at (3.5, 1.75) {};
             \node [style=none] (Zbl) at (3.5, -0.75) {};
	
		\node [style=inputdot] (5) at (3.1, -0.5) {};
		\node [style=empty] at (3.3, -0.5) {$5$};

		\node [style=empty] (Y) at (-2, 2) {$Y$};
		\node [style=none] (Ytr) at (-1.8, 1.75) {};
		\node [style=none] (Ytl) at (-2.4, 1.75) {};
		\node [style=none] (Ybr) at (-1.8, -0.75) {};
		\node [style=none] (Ybl) at (-2.4, -0.75) {};

		\node [style=inputdot] (2) at (-2, 1.5) {};
		\node [style=empty] at (-2.2, 1.5) {$2$};
		\node [style=inputdot] (3) at (-2, 0.5) {};
		\node [style=empty] at (-2.2, 0.5) {$3$};
		\node [style=inputdot] (4) at (-2, -0.5) {};
		\node [style=empty] at (-2.2, -0.5) {$4$};		
		
%		\node [style=empty] (Z) at (3, 1) {$Z$};
%		\node [style=none] (Ztr) at (3.25, 0.75) {};
%		\node [style=none] (Ztl) at (2.75, 0.75) {};
%		\node [style=none] (Zbl) at (2.75, -0.75) {};
%		\node [style=none] (Zbr) at (3.25, -0.75) {};
		
	\end{pgfonlayer}
	\begin{pgfonlayer}{edgelayer}
		\draw [style=inarrow,bend left=30, looseness=1.00] (B) to (c);
		\draw [style=inarrow, bend left=30, looseness=1.00] (c) to (C);
		\draw [style=inarrow] (D) to (d);
		\draw [style=inarrow] (d) to (E);
		\draw [style=inputarrow] (2) to (B);
		\draw [style=inputarrow] (3) to (C);
		\draw [style=inputarrow] (4) to (D);
		\draw [style=inputarrow] (5) to (E);
	
		\draw [style=simple] (Ytl.center) to (Ytr.center);
		\draw [style=simple] (Ytr.center) to (Ybr.center);
		\draw [style=simple] (Ybr.center) to (Ybl.center);
		\draw [style=simple] (Ybl.center) to (Ytl.center);
		\draw [style=simple] (Ztl.center) to (Ztr.center);
		\draw [style=simple] (Ztr.center) to (Zbr.center);
		\draw [style=simple] (Zbr.center) to (Zbl.center);
		\draw [style=simple] (Zbl.center) to (Ztl.center);
	\end{pgfonlayer}
\end{tikzpicture}
\]
Then their composite, $Q \odot P \maps X \nrightarrow Z$, looks like this:
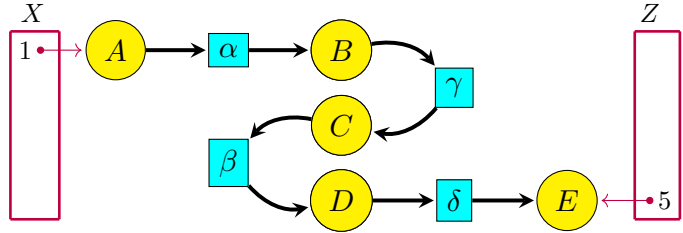
\[
\begin{tikzpicture}
	\begin{pgfonlayer}{nodelayer}
		\node [style=species] (A) at (-4, 1.5) {$A$};
		\node [style=species] (B) at (-1, 1.5) {$B$};
		\node [style=species] (C) at (-1, 0.5) {$C$};
		\node [style=species] (D) at (-1, -0.5) {$D$};
            \node [style=transition] (a) at (-2.5, 1.5) {$\alpha$}; 
             \node [style=transition] (b) at (-2.5, 0) {$\beta$}; 
		
		\node [style=empty] (X) at (-5.1, 2) {$X$};
		\node [style=none] (Xtr) at (-4.75, 1.75) {};
		\node [style=none] (Xbr) at (-4.75, -0.75) {};
		\node [style=none] (Xtl) at (-5.4, 1.75) {};
             \node [style=none] (Xbl) at (-5.4, -0.75) {};
	
		\node [style=inputdot] (1) at (-5, 1.5) {};
		\node [style=empty] at (-5.2, 1.5) {$1$};

		\node [style=species] (B) at (-1, 1.5) {$B$};
		\node [style=species] (C) at (-1, 0.5) {$C$};
		\node [style=species] (D) at (-1, -0.5) {$D$};
		\node [style=species] (E) at (2, -0.5) {$E$};
            \node [style=transition] (c) at (0.5, 1) {$\gamma$}; 
             \node [style=transition] (d) at (0.5, -0.5) {$\delta$}; 
		
		\node [style=empty] (Z) at (3.1, 2) {$Z$};
		\node [style=none] (Ztr) at (2.9, 1.75) {};
		\node [style=none] (Zbr) at (2.9, -0.75) {};
		\node [style=none] (Ztl) at (3.5, 1.75) {};
             \node [style=none] (Zbl) at (3.5, -0.75) {};
	
		\node [style=inputdot] (5) at (3.1, -0.5) {};
		\node [style=empty] at (3.3, -0.5) {$5$};

%		\node [style=empty] (Z) at (3, 1) {$Z$};
%		\node [style=none] (Ztr) at (3.25, 0.75) {};
%		\node [style=none] (Ztl) at (2.75, 0.75) {};
%		\node [style=none] (Zbl) at (2.75, -0.75) {};
%		\node [style=none] (Zbr) at (3.25, -0.75) {};
		
	\end{pgfonlayer}
	\begin{pgfonlayer}{edgelayer}
		\draw [style=inarrow] (A) to (a);
		\draw [style=inarrow] (a) to (B);
		\draw [style=inarrow,bend right=30, looseness=1.00] (C) to (b);
		\draw [style=inarrow, bend right=30, looseness=1.00] (b) to (D);
		\draw [style=inputarrow] (1) to (A);
	
		\draw [style=simple] (Xtl.center) to (Xtr.center);
		\draw [style=simple] (Xtr.center) to (Xbr.center);
		\draw [style=simple] (Xbr.center) to (Xbl.center);
		\draw [style=simple] (Xbl.center) to (Xtl.center);
		
		\draw [style=inarrow,bend left=30, looseness=1.00] (B) to (c);
		\draw [style=inarrow, bend left=30, looseness=1.00] (c) to (C);
		\draw [style=inarrow] (D) to (d);
		\draw [style=inarrow] (d) to (E);
		\draw [style=inputarrow] (5) to (E);
	
		\draw [style=simple] (Ztl.center) to (Ztr.center);
		\draw [style=simple] (Ztr.center) to (Zbr.center);
		\draw [style=simple] (Zbr.center) to (Zbl.center);
		\draw [style=simple] (Zbl.center) to (Ztl.center);
	\end{pgfonlayer}
\end{tikzpicture}
\]
We have 
\[   \blacksquare P = \{ (n,n,0,0) |\ n \in \N \} \subseteq \N \times \N^3 \]
since tokens starting at $A$ can only move to $B$, and similarly
\[   \blacksquare Q = \{ (0,0,n,n) |\ n \in \N \} \subseteq \N^3 \times \N  .\]
It follows that 
\[    \blacksquare Q \odot \blacksquare P = \{ (0,0) \} \subseteq \N \times \N .\]
On the other hand
\[    \blacksquare (Q \odot P) = \{ (n,n) |\ n \in \N \} \subseteq \N \times \N  \]
since in the composite open Petri net $QP$ tokens can move from $A$ to $E$.  The point is that tokens can only accomplish this by leaving the open Petri net $P$, going to $Q$, then returning to $P$, then going to $Q$.   The composite relation $\blacksquare Q \; \odot \blacksquare P$ only keeps track of processes where tokens leave $P$, move to $Q$, and never reenter $P$.

This makes it all the more impressive that the operational semantics
\[  \Open(F) \maps \Open(\Petri) \to \Open(\CMC) \]
is not lax:
\[    \Open(Q \odot P) \cong \Open(Q) \odot \Open(P)   .\]
We can see the difference in the example above: $\Open(Q) \odot \Open(P)$ contains a morphism $\delta \beta \gamma \alpha \maps A \to E$ which describes a process where tokens start in $P$, go to $Q$, then reenter $P$, and finally end in $Q$.   

On the other hand, the reachability semantics is maximally compatible with running Petri nets in parallel:

\begin{thm}
\label{thm:reachability_2}
The reachability semantics $\blacksquare \maps \Open(\Petri) \to \dRel$ is symmetric monoidal. 
\end{thm}
 
\begin{proof}
Because $\Cospan(F)$ is symmetric monoidal it suffices to show that 
\[  G \maps \Open(\CMC) \to \dRel \] 
is symmetric monoidal. This is simplified by that fact that $\dRel$ is a degenerate double category. Following Def.\ \ref{defn:monoidal_double_functor}, it suffices to show that
	\begin{itemize}
		\item $G_0 \maps (\Set,+) \to (\Set,\times)$ is symmetric monoidal,
		\item $G_1 \maps \Open(\CMC)_1 \to \dRel_1$ is symmetric monoidal, 
		\item we have equations of monoidal functors 
		\[  S \circ G_1 = G_0 \circ S , \qquad T \circ G_1 = G_0 \circ T,\]
		\item the composition and unit comparisons are monoidal natural transformations.	
	\end{itemize}
	
To show these things, first recall that $G_0 = \N = K \circ J$ where $K \colon \CMon \to \Set$ is the forgetful functor and $J \colon \Set \to \CMon$ is its left adjoint.   Since $J$ is a left adjoint it preserves
finite coproducts.   Since $K \maps \CMon \to \Set$ is a right adjoint is preserves finite products.  However, finite products in $\CMon$ are also finite coproducts.   Thus, $G_0$ maps finite coproducts to finite products, and is thus a symmetric monoidal functor from $(\Set,+)$ to $(\Set,\times)$.

Next, suppose we are given two open commutative monoidal categories
\[ \xymatrix{ L'X \ar[r]^i & C & L'Y, \ar[l]_o }  \qquad  \xymatrix{ L'X' \ar[r]^{i'} & C' & L'Y'. \ar[l]_{o'} }\]
Their tensor product is 
\[    \xymatrix{ L'(X+X') \ar[r]^{i+i'} & C+C' & L'(Y+Y') \ar[l]_{o+o'} }. \]
The set of objects of $L'(X+X')$ is naturally isomorphic to $\N[X] \times \N[X']$, and similarly
for $L'(Y+Y')$, so we have natural isomorphisms
\[     G_1(C+C') \cong \]
\[  \{ ((x,x',y,y') \in \N[X] \times \N[X'] \times \N[Y] \times \N[Y'] \, | \;
\exists h \maps i(x) \to o(y) \text{ and } \exists h' \maps i'(x') \to o'(y') \} \]
\[   \cong G_1(C) \times G_1(C') .\]
Using this fact one can check that $G_1$ is symmetric monoidal.
	
One can check that the equations $S \circ G_1=G_0 \circ S$ and $T \circ G_1 = G_0 \circ T$ are equations of monoidal functors, and the composition and unit comparisons of $G$ are trivially monoidal natural transformations because $\dRel$ is degenerate.
\end{proof}

\section{Conclusions}
\label{sec:conclusions}

The ideas presented here can be adapted to handle timed Petri nets, colored Petri nets with guards, and other kinds of Petri nets.   One can also develop a reachability semantics for open Petri nets that are glued together along transitions as well as places.  We hope to treat some of these generalizations in future work.

It would be valuable to have $\blacksquare (QP) = \blacksquare Q \odot \blacksquare P$, since then the reachability relation for an open Petri net could be computed compositionally, not merely `approximated from below' using $\blacksquare Q \odot \blacksquare P \subseteq \blacksquare(Q \odot P)$.   We conjecture that $\blacksquare (Q \odot P) = \blacksquare Q \odot \blacksquare P $ if $P$ and $Q$ are `one-way' open Petri nets.   Here an open Petri net
\[ \xymatrix{
					LX \ar[r]^i & P & LY \ar[l]_o } \]	
is \define{one-way} if no place in the image of $i$ appears in the target $t(\tau)$ of any transition
$\tau$ of $P$, and no place in the image of $o$ appears in the source $s(\tau)$ of any 
transition $\tau$ of $P$.   One-way open Petri nets should be the horizontal 1-cells in a full
sub-double category $\mathbb{O}\textbf{neWay}(\Petri)$ of $\Open(\Petri)$, and we conjecture that the reachability semantics restricts to an actual (not merely lax) double functor
\[    \blacksquare \maps \mathbb{O}\textbf{neWay}(\Petri) \to \dRel   .\]

\subsubsection*{Acknowledgements}

We would like to thank Kenny Courser for help with double categories and for a careful reading of this paper. We thank Christina Vasilakopoulou for spending many hours helping us figure out how to turn Petri nets into commutative monoidal categories.  We also thank Daniel Cicala, Joe Moeller, and Christian Williams for many insightful conversations.

\appendix
\section{Double Categories}
\label{appendix}

What follows is a brief introduction to double categories. A more detailed exposition can be found in the work of Grandis and Par\'e \cite{GP1,GP2}, and for monoidal double categories the work of Shulman \cite{Shulman2}.  We use `double category' to mean what earlier authors called a `pseudo
double category'.

\begin{defn}
\label{defn:double_category}
A \textbf{double category} is a category weakly internal to $\Cat$. More explicitly, a double category $\lD$ consists of:
\begin{itemize}
\item a \define{category of objects} $\lD_0$ and a \define{category of arrows} $\lD_1$,
\item  \define{source} and \define{target} functors
\[  S,T \colon \lD_1 \to \lD_0 ,\]
an \define{identity-assigning} functor
\[  U\colon \lD_0 \to \lD_1 ,\]
and a \define{composition} functor
\[ \odot \colon \lD_1 \times_{\lD_0} \lD_1 \to \lD_1 \]
where the pullback is taken over $\lD_1 \xrightarrow[]{T} \lD_0 \xleftarrow[]{S} \lD_1$,
such that
\[  S(U_{A})=A=T(U_{A}) , \quad
	S(M \odot N)=SN, \quad
   T(M \odot N)=TM, \]
\item natural isomorphisms called the \define{associator}
\[ \alpha_{N,N',N''} \maps (N \odot N') \odot N'' \xrightarrow{\sim} N \odot (N' \odot N'') , \]
the \define{left unitor}
\[		\lambda_N \maps U_{T(N)} \odot N \xrightarrow{\sim} N, \]
and the \define{right unitor}
\[  \rho_N \maps N \odot U_{S(N)} \xrightarrow{\sim} N \]
such that $S(\alpha), S(\lambda), S(\rho), T(\alpha), T(\lambda)$ and $T(\rho)$ are all identities and such that the standard coherence axioms hold: the pentagon identity for the 
associator and the triangle identity for the left and right unitor \cite[Sec.\ VII.1]{ML}.
\end{itemize}
If $\alpha$, $\lambda$ and $\rho$ are identities, we call $\lD$ a \define{strict} double category.
\end{defn}

Objects of $\lD_0$ are called \define{objects} and morphisms in $\lD_0$ are called \define{vertical 1-morphisms}.  Objects of $\lD_1$ are called \define{horizontal 1-cells} of $\lD$ and morphisms in $\lD_1$ are called \define{2-morphisms}.   A morphism $\alpha \maps M \to N$ in $\lD_1$ can be drawn as a square:
\[
\begin{tikzpicture}[scale=1]
\node (D) at (-4,0.5) {$A$};
\node (E) at (-2,0.5) {$B$};
\node (F) at (-4,-1) {$C$};
\node (A) at (-2,-1) {$D$};
\node (B) at (-3,-0.25) {$\Downarrow \alpha$};
\path[->,font=\scriptsize,>=angle 90]
(D) edge node [above]{$M$}(E)
(E) edge node [right]{$g$}(A)
(D) edge node [left]{$f$}(F)
(F) edge node [above]{$N$} (A);
\end{tikzpicture}
\]
where $f = S\alpha$ and $g = T\alpha$.  If $f$ and $g$ are identities we call $\alpha$ a \textbf{globular 2-morphism}.  These give rise to a bicategory:

\begin{defn}
\label{defn:horizontal}
Let $\lD$ be a double category. Then the $\textbf{horizontal bicategory}$ of $\lD$, denoted $H(\lD)$, is the bicategory consisting of objects, horizontal 1-cells and globular 2-morphisms of $\lD$.
\end{defn}

We have maps between double categories, and also transformations between maps:

\begin{defn}
\label{defn:double_functor}
Let $\lA$ and $\lB$ be double categories. A \textbf{double functor} $F \maps \lA \to \lB$ consists of:
\begin{itemize}
\item functors $F_0 \maps \lA_0 \to \lB_0$ and $F_1 \maps \lA_1 \to \lB_1$ obeying the following
equations: 
\[S \circ F_1 = F_0 \circ S, \qquad T \circ F_1 = F_0 \circ T,\]
\item natural isomorphisms called the \define{composition comparison}: 
\[   \phi(N,N') \maps F_1(N) \odot F_1(N') \stackrel{\sim}{\longrightarrow} F_1(N \odot N') \]
and the \define{identity comparison}:
\[  \phi_{A} \maps U_{F_0 (A)} \stackrel{\sim}{\longrightarrow} F_1(U_A) \]
whose components are globular 2-morphisms, 
\end{itemize}
such that the following diagram commmute:
\begin{itemize} 
\item a diagram expressing compatibility with the associator: 
\[\xymatrix{ 	(F_1(N) \odot F_1(N')) \odot F_1(N'') \ar[d]_{\phi (N,N') \odot 1} \ar[r]^{\alpha} & F_1(N) \odot (F_1(N') \odot F_1(N'')) \ar[d]^{1 \odot \phi(N',N'')} \\
			F_1(N \odot N') \odot F_1(N'') \ar[d]_{\phi(N \odot N', N'')} & F_1(N) \odot F_1(N' \odot N'') \ar[d]^{\phi(N, N'\odot N'')}\\
F_1((N \odot N') \odot N'') \ar[r]^{F_1(\alpha)} & F_1(N \odot (N' \odot N'')) }	\]
\item two diagrams expressing compatibility with the left and right unitors:
	\[
	\begin{tikzpicture}[scale=1.5]
	\node (A) at (1,1) {$F_1(N) \odot U_{F_0(A)}$};
	\node (A') at (1,0) {$F_1(N) \odot F_1(U_{A})$};
	\node (C) at (3.5,1) {$F_1(N)$};
	\node (C') at (3.5,0) {$F_1(N \odot U_A)$};
	\path[->,font=\scriptsize,>=angle 90]
	(A) edge node[left]{$1 \odot \phi_{A}$} (A')
	(C') edge node[right]{$F_1(\rho_N)$} (C)
	(A) edge node[above]{$\rho_{F_1(N)}$} (C)
	(A') edge node[above]{$\phi(N,U_{A})$} (C');
	\end{tikzpicture}
	\]
	\[
	\begin{tikzpicture}[scale=1.5]
	\node (B) at (5.5,1) {$U_{F_0(B)} \odot F_1(N)$};
	\node (B') at (5.5,0) {$F_1(U_{B}) \odot F_1(N)$};
	\node (D) at (8,1) {$F_1(N)$};
	\node (D') at (8,0) {$F_1(U_{B} \odot N).$};
		\path[->,font=\scriptsize,>=angle 90]
		(B) edge node[left]{$\phi_{B} \odot 1$} (B')
	(B') edge node[above]{$\phi(U_{B},N)$} (D')
	(B) edge node[above]{$\lambda_{F_1(N)}$} (D)
	(D') edge node[right]{$F_1(\lambda_{N})$} (D);
	\end{tikzpicture}
	\]
\end{itemize}
If the 2-morphisms $\phi(N,N')$ and $\phi_A$ are identities for all $N,N' \in \lA_1$ and 
$A \in \lA_0$, we say $F \maps \lA \to \lB$ is a \define{strict} double functor.  If on the other hand we drop the requirement that these 2-morphisms be invertible, we call $F$ a \define{lax} double
functor.
\end{defn}
	
\begin{defn}
Let $F \maps \lA \to \lB$ and $G \maps \lA \to \lB$ be lax double functors. A \define{transformation} $\beta \maps F \Rightarrow G$ consists of natural transformations $\beta_0 \maps F_0 \Rightarrow G_0$ and $\beta_1 \maps F_1 \Rightarrow G_1$ (both usually written as $\beta$) such that 
		\begin{itemize}
			\item $S( \beta_M) = \beta_{SM}$ and $T(\beta_M) = \beta_{TM}$ for any object $M \in \A_1$, 
			\item $\beta$ commutes with the composition comparison, and
			\item $\beta$ commutes with the identity comparison.
		\end{itemize}
\end{defn}
	
Shulman defines a 2-category $\mathbf{Dbl}$ of double categories, double functors, and transformations \cite{Shulman2}.  This has finite products.  In any 2-category with finite products we can define a pseudomonoid \cite{DayStreet}, which is a categorification of the concept of monoid.  For example, a pseudomonoid in $\mathsf{Cat}$ is a monoidal category.
	
\begin{defn}
	\label{defn:monoidal_double_category}
A \textbf{monoidal double category} is a pseudomonoid in $\mathbf{Dbl}$. Explicitly, a monoidal double category is a double category equipped with double functors $\otimes \maps \lD \times \lD \to \lD$ and $I \maps * \to \lD$ where $*$ is the terminal double category, along with invertible transformations called the \define{associator}:
\[  A \maps \otimes \, \circ \; (1_{\lD} \times \otimes ) \Rightarrow \otimes \; \circ \; (\otimes \times 1_{\lD}) ,\]
\define{left unitor}:
\[ L\maps \otimes \, \circ \; (1_{\lD} \times I) \Rightarrow 1_{\lD} ,\]
and \define{right unitor}:
\[ R \maps \otimes \,\circ\; (I \times 1_{\lD}) \Rightarrow 1_{\lD} \]
satisfying the pentagon axiom and triangle axioms.
\end{defn}

This definition neatly packages a large quantity of information.   Namely:
\begin{itemize}
\item $\lD_0$ and $\lD_1$ are both monoidal categories.
\item If $I$ is the monoidal unit of $\lD_0$, then $U_I$ is the
monoidal unit of $\lD_1$.
\item The functors $S$ and $T$ are strict monoidal.
\item $\otimes$ is equipped with composition and identity comparisons
\[ \chi \maps (M_1\otimes N_1)\odot (M_2\otimes N_2) \stackrel{\sim}{\longrightarrow}
(M_1\odot M_2)\otimes (N_1\odot N_2)\]
\[ \mu \maps U_{A\otimes B} \stackrel{\sim}{\longrightarrow} (U_A \otimes U_B)\]
making three diagrams commute as in Def.\ \ref{defn:double_functor}.
%		\[\xymatrix{
%			((M_1\otimes N_1)\odot (M_2\otimes N_2)) \odot (M_3\otimes N_3) \ar[r]^{\fx \odot 1} \ar[d]_{\alpha}
%			& ((M_1\odot M_2)\otimes (N_1\odot N_2)) \odot (M_3\otimes N_3) \ar[d]^{\fx}\\
%			(M_1\otimes N_1)\odot ((M_2\otimes N_2) \odot (M_3\otimes N_3)) \ar[d]_{1 \odot \fx} &
%			((M_1\odot M_2)\odot M_3) \otimes ((N_1\odot N_2)\odot N_3) \ar[d]^{\alpha \otimes \alpha}\\
%			(M_1\otimes N_1) \odot ((M_2\odot M_3) \otimes (N_2\odot N_3))\ar[r]^{\fx} &
%			(M_1\odot (M_2\odot M_3)) \otimes (N_1\odot (N_2\odot N_3))}\]
%		\[\xymatrix{(M\otimes N) \odot U_{C\otimes D} \ar[r]^{1 \odot \fu} \ar[d]_{\rho} &
%			(M\otimes N)\odot (U_C\otimes U_D) \ar[d]^{\fx}\\
%			M\otimes N\ar@{<-}[r]^{\rho \otimes \rho} & (M\odot U_C) \otimes (N\odot U_D)}\]
%		\[\xymatrix{U_{A\otimes B}\odot (M\otimes N)  \ar[r]^{\fu \odot 1} \ar[d]_{\lambda} &
%			(U_A\otimes U_B)\odot (M\otimes N) \ar[d]^{\fx}\\
%			M\otimes N\ar@{<-}[r]^{\lambda \otimes \lambda} & (U_A \odot M) \otimes (U_B\odot N)}\]

\item The associativity isomorphism for $\otimes$ is a transformation between double functors.
%		\[\xymatrix{
%			((M_1\otimes N_1)\otimes P_1) \odot ((M_2\otimes N_2)\otimes P_2) \ar[r]^{a \odot a}\ar[d]_{\fx} &
%			(M_1\otimes (N_1\otimes P_1)) \odot (M_2\otimes (N_2\otimes P_2)) \ar[d]^{\fx}\\
%			((M_1\otimes N_1) \odot (M_2\otimes N_2)) \otimes (P_1\odot P_2) \ar[d]_{\fx \otimes 1} &
%			(M_1\odot M_2) \otimes ((N_1\otimes P_1)\odot (N_2\otimes P_2))\ar[d]^{1 \otimes \fx} \\
%			((M_1\odot M_2) \otimes(N_1\odot N_2)) \otimes (P_1\odot P_2) \ar[r]^{a} &
%			(M_1\odot M_2) \otimes ((N_1\odot N_2)\otimes (P_1\odot P_2))}\]
%		\[\xymatrix{
%			U_{(A\otimes B)\otimes C} \ar[r]^{U_{a}} \ar[d]_{\fu} & U_{A\otimes (B\otimes C)} \ar[d]^{\fu}\\
%			U_{A\otimes B} \otimes U_C \ar[d]_{\fu \otimes 1} & U_A\otimes U_{B\otimes C}\ar[d]^{1 \otimes \fu}\\
%			(U_A\otimes U_B)\otimes U_C \ar[r]^{a} & U_A\otimes (U_B\otimes U_C) }\]
		\item The unit isomorphisms are transformations
between double functors.
%		\[\vcenter{\xymatrix{
%				(M\otimes U_I)\odot (N\otimes U_I)\ar[r]^{\fx}\ar[d]_{r \odot r} &
%				(M\odot N)\otimes (U_I \odot U_I) \ar[d]^{1 \otimes \rho}\\
%				M\odot N \ar@{<-}[r]^{r} &
%				(M\odot N)\otimes U_I }}\]
%		\[\vcenter{\xymatrix{U_{A\otimes I} \ar[r]^{\fu} \ar[dr]_{U_{r}} & U_A\otimes U_I \ar[d]^{r}\\
%				& U_A}}\]
%		\[\vcenter{\xymatrix{
%				(U_I\otimes M)\odot (U_I\otimes N)\ar[r]^{\fx} \ar[d]_{\ell \odot \ell} &
%				(U_I \odot U_I) \otimes (M\odot N) \ar[d]^{\lambda \otimes 1}\\
%				M\odot N \ar@{<-}[r]^{\ell} &
%				U_I\otimes (M\odot N) }}\]
%		\[\vcenter{\xymatrix{U_{I\otimes A} \ar[r]^{\fu}\ar[dr]_{U_{\ell}} & U_I\otimes U_A \ar[d]^{\ell}\\
%				& U_A}}\]
%		\newcounter{mondbl}
%		\setcounter{mondbl}{\value{enumi}}
	\end{itemize}

	\begin{defn}
	\label{defn:symmetric_monoidal_double_category}
A \define{braided monoidal double category} is a monoidal double
category equipped with an invertible transformation
\[ \beta \maps \otimes \Rightarrow \otimes \circ \tau \]
called the \define{braiding}, where $\tau \maps \lD \times \lD \to \lD \times \lD$ is the twist double functor sending pairs in the object and arrow categories to the same pairs in the opposite order. The braiding is required to satisfy the usual two hexagon identities \cite[Sec.\ XI.1]{ML}.  If the braiding is self-inverse we say that $\lD$ is a \define{symmetric monoidal double category}.
	\end{defn}
	
In other words:
\begin{itemize}
		\item $\lD_0$ and $\lD_1$ are braided (resp. symmetric) monoidal categories,
		\item the functors $S$ and $T$ are strict braided monoidal functors, and
		\item the braiding is a transformation between double functors.
\end{itemize}

\begin{defn}
\label{defn:monoidal_double_functor}
A \define{monoidal lax double functor} $F \colon \lC \to \lD$ between monoidal double categories $\lC$ and $\lD$ is a lax double functor $F \maps \lC \to \lD$ such that
	\begin{itemize}
		\item $F_0$ and $F_1$ are monoidal functors,
		\item $SF_1= F_0S$ and $TF_1 = F_0T$ are equations between monoidal functors, and
		\item the composition and unit comparisons $\phi(N_1,N_2) \maps F_1(N_1) \odot F_1(N_2) \to F_1(N_1\odot N_2)$ and $\phi_A \maps U_{F_0 (A)} \to F_1(U_A)$ are monoidal natural transformations.
	\end{itemize}
The monoidal lax double functor is \define{braided} if $F_0$ and $F_1$ are braided monoidal functors and \define{symmetric} if they are symmetric monoidal functors. 
\end{defn}

\end{document}